\documentclass[12pt, twoside, leqno]{article}

\usepackage{amsmath,amsthm}
\usepackage{amssymb}

\usepackage{enumitem}

\usepackage{graphicx}

\usepackage[T1]{fontenc}
\usepackage[T1]{fontenc}
\usepackage{amsfonts}
\usepackage{amsmath,amsthm}
\usepackage{amssymb,latexsym}
\usepackage{tikz}
\usetikzlibrary{arrows,shapes,decorations,decorations.text}
\usepackage{xcolor}

\theoremstyle{plain}
\newtheorem{theorem}{Theorem}[section]

\newtheorem{proposition}[theorem]{Proposition}
\newtheorem{corollary}[theorem]{Corollary}

\theoremstyle{definition}
\newtheorem{example}[theorem]{Example}
\newtheorem{definition}[theorem]{Definition}
\theoremstyle{remark}
\newtheorem{remark}[theorem]{Remark}
\newtheorem{question}[theorem]{Question}

\DeclareMathOperator{\inter}{{\rm int}}
\DeclareMathOperator{\cl}{{\rm cl}}

\theoremstyle{definition}

\numberwithin{equation}{section}

\begin{document}


\baselineskip=17pt


\title{$P$-spaces in the absence of the Axiom of Choice}

\author{Kyriakos Keremedis\\
Department of Mathematics, University of the Aegean\\ 
Karlovassi 83200, Samos, Greece\\
E-mail: kker@aegean.gr
\and 
AliReza Olfati\\
Department of Mathematics. Faculty of Basic Sciences,\\ 
Yasouj University, Daneshjoo St.,  Yasouj 75918-74934, Iran\\
E-mail: alireza.olfati@yu.ac.ir
\and 
Eliza Wajch\\
Institute of Mathematics
Faculty of Exact and Natural Sciences,\\
Siedlce University of Natural Sciences and Humanities,\\
ul. 3 Maja 54, 08-110 Siedlce, Poland\\
E-mail: eliza.wajch@gmail.com}

\date{}

\maketitle


\renewcommand{\thefootnote}{}



\renewcommand{\thefootnote}{\arabic{footnote}}
\setcounter{footnote}{0}


\begin{abstract}
A $P$-space is a topological space whose every $G_{\delta}$-set is open. In this article, basic properties of $P$-spaces are investigated in the absence of the Axiom of Choice. New weaker forms of the Axiom of Choice, all relevant to $P$-spaces or to countable intersections of $G_{\delta}$-sets, are introduced. Several independence results are obtained and open problems are posed. It is shown that a zero-dimensional subspace of the real line may fail to be strongly zero-dimensional in $\mathbf{ZF}$. Among the open problems there is the question whether it is provable in $\mathbf{ZF}$ that every finite product of $P$-spaces is a $P$-space. A partial answer to this question is given. 
\end{abstract}

\section{Introduction}
\label{s1}
\subsection{The main aim}
\label{s1.1}
In \cite{gh}, L. Gillman and M. Henriksen introduced the concepts of a $P$-point and a $P$-space in the class of Tychonoff spaces. These concepts can be extended to the class of all topological spaces as follows:  

\begin{definition}
	\label{s1d1}
	(Cf. \cite{gh}.) Let $\mathbf{X}=\langle X, \tau\rangle$ be a topological space. 
	\begin{enumerate}
		\item[(i)] A point $x\in X$ is called a \emph{$P$-point} of $\mathbf{X}$ if, for every $G_{\delta}$-set $U$ in $\mathbf{X}$ such that $x\in U$, it holds that $x\in\inter_{\mathbf{X}}(U)$  where  $\inter_{\mathbf{X}}(U)$ denotes the interior of $U$ in $\mathbf{X}$.
		\item[(ii)] $\mathbf{X}$ is called a \emph{$P$-space} (in the sense of Gillman-Henriksen) if every point of $X$ is a $P$-point in $\mathbf{X}$.
	\end{enumerate}
\end{definition}

Tychonoff $P$-spaces have been widely studied in $\mathbf{ZFC}$ by many authors and are still being investigated (cf., e.g., \cite{Dow}, \cite{gh}, \cite{hart}, \cite{mis}); however, they do not seem to have been ever studied in deep in the absence of the Axiom of Choice.
 
In this article, the intended context of reasoning and statements of theorems is the Zermelo-Fraenkel set theory $\mathbf{ZF}$ in which the Axiom of Choice (denoted by $\mathbf{AC}$) is deleted. The main aim of this article is to establish which basic properties of (not necessarily Tychonoff) $P$-spaces that were proved in $\mathbf{ZFC}$ in the past are also provable in $\mathbf{ZF}$, and to show several statements on $P$-spaces that are independent of $\mathbf{ZF}$ or follow in $\mathbf{ZF}$ from weaker assumptions than $\mathbf{AC}$. Since Tychonoff $P$-spaces can be characterized in $\mathbf{ZFC}$ by some properties of their algebras of continuous real functions (for instance, in view of \cite[Theorem 4.2]{gh}, a Tychonoff space $\mathbf{X}$ is a $P$-space if and only if every continuous real function on $\mathbf{X}$ is locally constant), we also pay attention to some special subsets of the algebras of continuous real functions. 

Unfortunately, we have not verified yet everything what is known about $P$-spaces in $\mathbf{ZFC}$ to see what remains from it in $\mathbf{ZF}$. However, many new results of our first study of $P$-spaces in $\mathbf{ZF}$ are included in this article. We leave more extensive study in this direction for the future. 

Before we pass to the body of this article, let us lay out our notation and basic terminology in  Sections \ref{s1.2}--\ref{s1.4}, and describe in brief the content of the article in Section \ref{s1.5}.

\subsection{Set-theoretic terminology and notation}
\label{s1.2}

To stress the fact that a result is proved in $\mathbf{ZF}$ or $\mathbf{ZF+\Psi}$ (where $\mathbf{\Psi}$ is a statement independent of $\mathbf{ZF}$), we shall write
at the beginning of the statements of the theorems and propositions [$%
\mathbf{ZF}$] or [$\mathbf{ZF+\Psi}$], respectively. Apart from models of $\mathbf{ZF}$, we refer to some models
of $\mathbf{ZFA}$ (or $\text{ZF}^0$ in \cite{hr}), that is, we refer also to Zermelo-Fraenkel set theory with atoms (see \cite{j}, \cite{j1} and \cite{hr}). 

We recall that a \emph{Dedekind-infinite} set is any set $X$ which is equipotent to a proper subset of $X$. If a set is not Dedekind-infinite, it is called \emph{Dedekind-finite}. As usual, $\omega$ is the set of all Dedekind-finite ordinal numbers of von Neumann. A set equipotent to a member of $\omega$ is called \emph{finite}. A set is \emph{infinite} if it is not finite. A set $X$ is \emph{countable} if $X$ is equipotent to a subset of $\omega$. An infinite countable set is called \emph{denumerable}. Members of $\omega$ are non-negative integers. For $n\in\omega$, $n+1=n\cup\{n\}$. For our convenience, we put $\mathbb{N}=\omega\setminus\{0\}$. 

For every set $X$, $\mathcal{P}(X)$ denotes the power set of $X$, $[X]^{<\omega}$ denotes  the set of all finite subsets of $X$, and $[X]^{\leq \omega}$ stands for the family of all countable subsets of $X$.  A \emph{partition} of a set $X$ is a pairwise disjoint cover of $X$. Let us recall the following notions: 

\begin{definition}
	\label{s1d2}
	A set $X$ is called
	\begin{enumerate} 
		\item[(i)] \emph{weakly Dedekind-infinite} if $\mathcal{P}(X)$ is  Dedekind-infinite (otherwise, $X$ is called \emph{weakly Dedekind-finite}) (cf., e.g., \cite{kw2} and \cite[Note 94]{hr});
		\item[(ii)] \emph{quasi Dedekind-infinite} if the set $[X]^{<\omega}$ is Dedekind-infinite (otherwise, $X$ is called \emph{quasi Dedekind-finite}) (cf., e.g., \cite{kw2});
		\item[(iii)] \emph{amorphous} if $X$ is infinite and, for every infinite subset $A$ of $X$, the set $X\setminus A$ is finite (cf., e.g., \cite[Note 57]{hr});
		\item[(iv)] a \emph{cuf} set if $X$ is a countable union of finite sets (cf., e.g., \cite{kw3});
		\item[(v)] a \emph{cuc} set if $X$ is a countable union of countable sets.
	\end{enumerate}
\end{definition}

\begin{remark}
	\label{s1r3}
	It is well known  and easy to observe that every amorphous set is weakly Dedekind-finite, every weakly Dedekind-finite set is quasi Dedekind-finite, every Dedekind-infinite set is quasi Dedekind-infinite.  
\end{remark}

Given a family $\mathcal{A}=\{A_j: j\in J\}$, every function $f\in\prod\limits_{j\in J}A_j$ is called a \emph{choice function} of $\mathcal{A}$. If $I$ is an infinite subset of $J$, then a choice function of $\{A_j: j\in I\}$ is called a \emph{partial choice function} of $\mathcal{A}$. Every choice function of $\{[A_j]^{<\omega}\setminus\{\emptyset\}: j\in J\}$ is called a \emph{multiple choice function} of $\mathcal{A}$. For sets $J$ and $X$, if for every $j\in J$, $X_j=X$, then $X^{J}=\prod\limits_{j\in J}X_j$.

Let us recall definitions and abbreviations of several forms (all weaker than $\mathbf{AC}$) which will be used in the sequel.

\begin{definition}
	\label{hrforms}
	\begin{enumerate}
		\item $\mathbf{MC}$ (\cite[Form 67]{hr}): Every non-empty family of non-empty sets has a multiple choice function.
		\item $\mathbf{CAC}$ (\cite[Form 8]{hr}): Every denumerable family of non-empty sets has a choice function.
		\item $\mathbf{CAC}(\mathbb{R})$ (\cite[Form 94]{hr}): Every denumerable family of non-empty subsets of $\mathbb{R}$ has a choice function.
		\item $\mathbf{CAC}_{\omega}(\mathbb{R})$ (\cite[Form 5]{hr}): Every denumerable family of denumerable subsets of $\mathbb{R}$ has a choice function.
		\item $\mathbf{CAC}_{fin}$ (\cite[Form 10]{hr}): Every denumerable family of non-empty finite sets has a choice function.
		\item $\mathbf{CMC}$ (\cite[Form 126]{hr}): Every denumerable family of non-empty sets has a multiple choice function.
		\item $\mathbf{CMC}_{\omega}$ (\cite[Form 350]{hr}): Every denumerable family of denumerable sets has a multiple choice function.
		\item $\mathbf{CMC}(\aleph_0, \infty)$ (\cite[Form 131]{hr}): For every family $\{A_{i}: i\in\omega\}$ of non-empty sets there exists a family $\{B_{i}: i\in\omega\}$ of non-empty countable sets such that, for every $i\in\omega$, $B_{i}\subseteq A_{i}$.
		
		\item $\mathbf{DC}$ (the Principle of Dependent Choices, \cite[Form 43]{hr}): For every set $X$ and every binary relation $R\subseteq X\times X$ such that $(\forall x\in X)(\exists y\in X)\langle x, y\rangle\in R$, there exists $f\in X^{\omega}$ such that $(\forall n\in\omega)\langle f(n), f(n+1)\rangle\in R$.
		
		\item $\mathbf{WDC}$ (the Principle of Weak Dependent Choices, \cite[Form \text{[106 B]}]{hr}): If $R$ is a binary relation on a set $E$ such that $(\forall x\in E)(\exists y\in E)(\langle x, y\rangle\in R)$, then there exists a sequence $(F_n)_{n\in\omega}$ of non-empty finite subsets of $E$ such that $(\forall n\in\omega)(\forall x\in F_n)(\exists y\in F_{n+1})(\langle x, y\rangle\in R)$. 
		
		\item $\mathbf{vDCP}(\omega)$ (van Douwen's Choice Principle, \cite[Form 119]{hr}): For every disjoint family  $\mathcal{A}=\{A_i: i\in\omega\}$ such that there exists $f\in\prod\limits_{i\in \omega}\mathcal{P}(A_i\times A_i)$  such that, for every $i\in\omega$, the pair $\langle A_i, f(i)\rangle$ is a linearly ordered set order isomorphic to $\langle\mathbb{Z}, \leq\rangle$ where $\leq$ is the standard linear order in $\mathbb{Z}$, it holds that $\mathcal{A}$ has a choice function.

		\item $\mathbf{IDI}$ (\cite[Form 9]{hr}): Every infinite set is Dedekind-infinite.
		\item $\mathbf{IQDI}$ (\cite{kw2}): Every infinite set is quasi Dedekind-infinite.
		\item $\mathbf{IWDI}$ (\cite[Form 82]{hr}): Every infinite set is weakly Dedekind-infinite.
		\item $\mathbf{NAS}$ (\cite[Form 64]{hr}): There are no amorphous sets. 
		\item $\mathbf{IDI}(\mathbb{R})$ (\cite[Form 13]{hr}): Every infinite subset of $\mathbb{R}$ is Dedekind-infinite.
		\item $\mathbf{UT}(\aleph_0, cuf, cuf)$ (\cite[Form 419]{hr1}): Every countable union of cuf sets is a cuf set. 
	\end{enumerate}
\end{definition}

\begin{remark}
	\label{s1r5}
	
	(i) It is obvious that the following implications are true in $\mathbf{ZF}$: 
	$$\mathbf{IDI}\rightarrow\mathbf{IQDI}\rightarrow \mathbf{IWDI}\rightarrow \mathbf{NAS}.$$
	None of the implications given above is reversible in $\mathbf{ZF}$ (see \cite{hr} and \cite{kw2}).
	
	(ii) $\mathbf{WDC}$ is equivalent to the Principle of Dependent Multiple Choices (see \cite[Form \text{[106 A]}]{hr}) whose definition in \cite{hr} is more complicated than the definition of $\mathbf{WDC}$. Let us denote by $\mathbf{DMC}$ the Principle of Dependent Multiple Choices. 
\end{remark}

In the sequel, for sets $X$ and $Y$, $|X|\leq |Y|$ is the abbreviation to the sentence: ``$X$ is equipotent to a subset of $Y$''. If $|X|\leq|Y|$, we say that $X$ is of size $\leq |Y|$. In $\mathbf{ZF}$,  we need not define cardinalities of sets but we can define well-ordered cardinals as initial ordinal numbers of von Neumann. When we refer to $\omega$ as to a well-ordered cardinal, we denote it by $\aleph_0$. If $\kappa$ is a well-ordered cardinal and $X$ is a set, then $|X|\leq\kappa$ means: ``$X$ is equipotent to a subset of $\kappa$''.

\begin{definition}
	\label{s1d6}
	Let $\mathcal{A}$ be a family of subsets of a set $X$. Then:
	\begin{enumerate}
		\item[(i)] $\mathcal{A}_{\delta}=\{\bigcap \mathcal{U}: \mathcal{U}\in[\mathcal{A}]^{\leq\omega}\setminus\{\emptyset\}\}$;
		\item[(ii)] $\mathcal{A}_{\sigma}=\{\bigcup \mathcal{U}: \mathcal{U}\in[\mathcal{A}]^{\leq\omega}\}$;
		\item[(iii)] $\mathcal{A}$ is \emph{stable under countable intersections} if $\mathcal{A}_{\delta}\subseteq\mathcal{A}$;
		\item[(iv)] $\mathcal{A}$ is \emph{stable under countable unions} if $\mathcal{A}_{\sigma}\subseteq\mathcal{A}$;
		\item[(v)] $\mathcal{A}$ is \emph{stable under finite intersections} (respectively, \emph{finite unions}) if, for every pair $A_1, A_2$ of members of $\mathcal{A}$, $A_1\cap A_2\in\mathcal{A}$ (respectively, $A_1\cup A_2\in\mathcal{A}$);
		\item[(vi)] $\mathcal{A}$ is \emph{centered} if $\mathcal{A}\neq\emptyset$ and, for every non-empty finite subfamily $\mathcal{U}$ of $\mathcal{A}$, $\bigcap\mathcal{U}\neq\emptyset$.
	\end{enumerate}
\end{definition}

In the following definition, we introduce several new forms which will be used in the forthcoming sections:

\begin{definition}
	\label{s1d7}
	\begin{enumerate}
		\item $\mathbf{UT}(\aleph_0, cuc, cuc)$: Every countable union of cuc sets is a cuc set.
		\item $\mathbf{CACCLO}$: For every family $\mathcal{A}=\{A_i: i\in\omega\}$ of non-empty sets such that there exists $f\in\prod\limits_{i\in\omega}\mathcal{P}(A_i\times A_i)$ such that, for every $i\in\omega$, $f(i)$ is a strongly complete linear ordering of $A_i$ (that is, $f(i)$ is a linear ordering of $A_i$ such that every non-empty countable subset of $A_i$ has a maximum in $\langle A_i, f(i)\rangle$), it holds that $\mathcal{A}$ has a choice function. 
		\item $\mathbf{PCACCLO}$: For every family $\mathcal{A}=\{A_i: i\in\omega\}$ of non-empty sets such that there exists $f\in\prod\limits_{i\in\omega}\mathcal{P}(A_i\times A_i)$ such that, for every $i\in\omega$, $f(i)$ is a strongly complete linear ordering of $A_i$, it holds that $\mathcal{A}$ has a partial choice function. 
		\item $\mathbf{A}(\sigma\sigma)$ (respectively, $\mathbf{A}(\delta\delta)$): For every set $X$ and every family $\mathcal{A}$ of subsets of $X$ such that $\mathcal{A}$ is stable under finite unions and finite intersections, for every family $\{\mathcal{U}_n: n\in\omega\}$ of non-empty countable subfamilies of $\mathcal{A}$, there exists a family $\{V_{n,m}: n,m\in\omega\}$ of members of $\mathcal{A}$ such that, for every $n\in\omega$, $\bigcup\mathcal{U}_n=\bigcup\limits_{m\in\omega}V_{n,m}$ (respectively, $ \bigcap\mathcal{U}_n=\bigcap\limits_{m\in\omega}V_{n,m}$).
	\end{enumerate}
\end{definition}

\subsection{Preliminaries related to topology}
\label{s1.3}

We usually denote topological or metric spaces with boldface letters, and their underlying sets with lightface letters. However, sometimes,  for a topological space $\mathbf{X}=\langle X, \tau\rangle$ and a set $Y\subseteq X$, we denote by $Y$ the topological space $\mathbf{Y}=\langle X, \tau|_Y\rangle$, where $\tau|_Y=\{U\cap Y: U\in\tau\}$ (that is, subsets of a given topological space $\mathbf{X}$ are considered as topological subspaces of $\mathbf{X}$). For a subset $A$ of a topological space $\mathbf{X}=\langle X, \tau\rangle$, the closure of $A$ in $\mathbf{X}$ is  denoted by $\cl_{\mathbf{X}}(A)$, and the interior of $A$ in $\mathbf{X}$ is denoted by $\inter_{\mathbf{X}}(A)$ (as in Definition \ref{s1d1}(i)).

We denote by $\tau_{nat}$ the natural topology of the real line $\mathbb{R}$, and the space $\langle \mathbb{R}, \tau_{nat}\rangle$ by $\mathbb{R}$ if this is not misleading. We also denote by $\mathbb{R}$ the metric space $\langle \mathbb{R}, \rho_e\rangle$ where, for all $x,y\in\mathbb{R}$,  $\rho_e(x,y)=|x-y|$. If not stated otherwise, all subsets of $\mathbb{R}$ are considered as topological subspaces of $\langle \mathbb{R}, \tau_{nat}\rangle$ and as a metric subspaces of $\langle\mathbb{R}, \rho_e\rangle$.

For every set $Y$, the discrete space $\langle Y, \mathcal{P}(Y)\rangle$ is denoted by $Y_{disc}$. In particular, $\mathbb{R}_{disc}$ denotes the discrete space $\langle \mathbb{R}, \mathcal{P}(\mathbb{R})\rangle$. The discrete space $\langle \{0,1\}, \mathcal{P}(\{0,1\})\rangle$ is denoted by $\mathbf{2}$.

Given a non-empty set $J$ and a family $\{\langle X_j, \tau_j\rangle: j\in J\}$ of topological spaces, if we write $\mathbf{X}=\prod\limits_{j\in J}\langle X_j,\tau_j\rangle$, then $\mathbf{X}=\langle X, \tau\rangle$ where $X=\prod\limits_{j\in J}X_j$ and $\tau$ is the (Tychonoff) product topology on $X$ determined by the family $\{\tau_j: j\in J\}$. For a fixed $j_0\in J$, $\pi_{j_0}:\prod\limits_{j\in J}X_j\to X_{j_0}$ denotes the projection defined by: for every $x\in\prod\limits_{j\in J}X_j$, $\pi_{j_0}(x)=x(j_0)$. 

For sets $X$ and $Y$, $\pi_X: X\times Y\to X$ and $\pi_Y: X\times Y\to Y$ denote the standard projections.

For topological spaces $\mathbf{X}=\langle X, \tau_X\rangle$ and $\mathbf{Y}=\langle Y, \tau_Y\rangle$, $C(\mathbf{X},\mathbf{Y})$ stands for the set of all $\langle \tau_X, \tau_Y\rangle$-continuous mappings from $X$ to $Y$. If this is not misleading, instead of ``$\langle \tau_X, \tau_Y\rangle$-continuous mappings'', we write shortly ``continuous mappings'' .

Let  $\mathbf{Y}=\langle Y, \rho\rangle$ be a given metric space. For every $y\in Y$ and a positive real number $\varepsilon$, 
$$B_{\rho}(y, \varepsilon)=\{z\in Y: \rho(y,z)<\varepsilon\}.$$
We denote by $\tau(\rho)$ the topology on $Y$ having $\{B_{\rho}(y, \varepsilon): y\in Y, \varepsilon>0\}$ as a base. The topology $\tau(\rho)$ is usually called \emph{induced by} $\rho$. For $A\subseteq Y$, let $\delta_{\rho}(A)=\sup\{ \rho(x,y): x,y\in A\}$. 

Now, let  $\mathbf{X}=\langle X, \tau\rangle$ be a given topological space, and $\mathbf{Y}=\langle Y, \rho\rangle$ a given metric space. We denote by $C(\mathbf{X},\mathbf{Y})$ the set of all $\langle \tau, \tau(\rho)\rangle$-continuous functions from $X$ to $Y$. We define a metric $d_u$ on $C(\mathbf{X},\mathbf{Y})$ as follows:
$$d_u(f,g)=\sup\{ \min\{\rho (f(x), g(x)), 1)\}: x\in X\}$$
for all $f,g\in C(X,Y)$. Then  the topological space $\langle C(\mathbf{X},\mathbf{Y}), \tau(d_u)\rangle$ is denoted by $C_u(\mathbf{X},\mathbf{Y})$.  For $f\in C(\mathbf{X},\mathbf{Y})$ and $A\subseteq X$, the oscillation of $f$ on $A$ is defined by:
$$\text{osc}_A(f)=\sup\{\rho(f(x), f(y)): x,y\in A\}.$$

The set $C(\mathbf{X},\mathbb{R})$ is denoted by $C(\mathbf{X})$, and the space $C_u(\mathbf{X}, \mathbb{R})$ by $C_u(\mathbf{X})$. 

For $f\in C(\mathbf{X})$, let $Z(f)=f^{-1}[\{0\}]$. Then $Z(f)$ is the \emph{zero-set} (or 0-\emph{set}) determined by $f$. Let $\mathcal{Z}(\mathbf{X})=\{Z(f): f\in C(\mathbf{X})\}$. The members of $\mathcal{Z}(\mathbf{X})$ are called \emph{zero-sets} (or 0-\emph{sets}) of $\mathbf{\mathbf{X}}$, their complements are called \emph{co-zero sets} of $\mathbf{\mathbf{X}}$. The family of all co-zero sets of $\mathbf{X}$ is denoted by $\mathcal{Z}^{c}(\mathbf{X})$.  The family of all clopen subsets of a topological space $\mathbf{X}$ is denoted by $\mathcal{CO}(\mathbf{X})$. The family of all $G_{\delta}$-sets of $\mathbf{X}$ is denoted by $\mathcal{G}_{\delta}(\mathbf{X})$. The family of all $F_{\sigma}$-sets of $\mathbf{X}$ is denoted by $\mathcal{F}_{\sigma}(\mathbf{X})$.

\begin{definition}
	\label{s1d8}
	For a topological space $\mathbf{X}$, we let $\mathcal{CO}_{\delta}(\mathbf{X})=\mathcal{CO}(\mathbf{X})_{\delta}$ and $\mathcal{Z}_{\delta}(\mathbf{X})=\mathcal{Z}(\mathbf{X})_{\delta}$. Every member of $\mathcal{CO}_{\delta}(\mathbf{X})$ is called a $c_{\delta}$-set in $\mathbf{X}$. Every member of $\mathcal{Z}_{\delta}(\mathbf{X})$ is called a $z_{\delta}$-set in $\mathbf{X}$.
\end{definition}

\begin{remark}
	\label{s1r9}
	For every topological space $\mathbf{X}$, the following inclusions hold: $$\mathcal{CO}(\mathbf{X})\subseteq\mathcal{CO}_{\delta}(\mathbf{X})\subseteq\mathcal{Z}(\mathbf{X})\subseteq\mathcal{G}_{\delta}(\mathbf{X}).$$
The forthcoming Theorem \ref{s4t10} shows that, in $\mathbf{ZF}$, a $z_{\delta}$-set in $\mathbf{X}$ may fail to be a zero-set in $\mathbf{X}$. 		
\end{remark}

The sets of functions $A(\mathbf{X})$ and $U_{\aleph_0}(\mathbf{X})$, defined below, are especially important here:

\begin{definition} 
	\label{s1d10}
	If $\mathbf{X}$ is a topological space, then:
	\begin{enumerate}
		\item[(i)] $U_{\aleph_0}(\mathbf{X})$ is the set of all $f\in C(\mathbf{X})$ such that, for every positive real number $\varepsilon$, there exists a countable cover $\mathcal{A}$ of $\mathbf{X}$ such that $\mathcal{A}\subseteq\mathcal{CO}(\mathbf{X})$ and, for each $A\in\mathcal{A}$, $\text{osc}_A(f)\leq\varepsilon$;
		\item[(ii)] (Cf. \cite{bk}.) $A(\mathbf{X})$ is the set of all $f\in C(\mathbf{X})$ such that, for each  $O\in\tau_{nat}$, the set $f^{-1}[O]$ is a countable union of clopen sets of $\mathbf{X}$. 
	\end{enumerate}
\end{definition}

Every open base of a topological space $\mathbf{X}$ will be called a \emph{base} of $\mathbf{X}$.

\begin{definition}
	\label{s1d11} 
	We call a topological space $\mathbf{X}=\langle X, \tau\rangle$:
	\begin{enumerate}
		\item[(i)] \emph{regular} if for every $x\in X$ and every $V\in \tau$ such that $x\in V$, there exists $U\in \tau$ such that $x\in U$ and $\cl_{\mathbf{X}}(U)\subseteq V$;
		\item[(ii)] a \emph{$T_3$-space} if $\mathbf{X}$ is a regular $T_1$-space;
		\item[(iii)] \emph{completely regular} if $\mathcal{Z}^{c}(\mathbf{X})$ is a base of $\mathbf{X}$;
		\item[(iv)] \emph{Tychonoff} or $T_{3\frac{1}{2}}$-space if $\mathbf{X}$ is a completely regular $T_1$-space;
		\item[(v)] \emph{zero-dimensional} if $\mathcal{CO}(\mathbf{X})$ is a base of $\mathbf{X}$;
		\item[(vi)] \emph{strongly zero-dimensional} if $\mathbf{X}$ is completely regular and, for every pair $Z_1, Z_2$ of disjoint zero-sets in $\mathbf{X}$, there exists $U\in\mathcal{CO}(\mathbf{X})$ such that $Z_1\subseteq U\subseteq X\setminus Z_2$;
		\item [(vii)] \emph{Loeb} if the family of all non-empty closed sets of $\mathbf{X}$ has a choice function (cf., e.g., \cite{br}, \cite{kt}, \cite{kw4});
		\item[(viii)] \emph{limit point compact} if every infinite subset of $\mathbf{X}$ has an accumulation point in $\mathbf{X}$ (cf., e.g., \cite{klpc});
		\item [(ix)] \emph{(countably) compact} if every (countable) open cover of $\mathbf{X}$ has a finite subcover;
		\item[(x)] a \emph{Baire space} if, for every denumerable family $\mathcal{U}$ of dense open subsets of $\mathbf{X}$, the set $\bigcap\mathcal{U}$ is dense in $\mathbf{X}$;
		\item[(xi)] a \emph{quasi Baire space} if, for every denumerable family $\mathcal{U}$ of dense open subsets of $\mathbf{X}$, $\bigcap\mathcal{U}\neq\emptyset$;
		\item[(xii)] a \emph{strongly Baire space} if, for every denumerable family $\mathcal{U}$ of dense open subsets of $\mathbf{X}$, the set $\bigcap\mathcal{U}$ is open and dense in $\mathbf{X}$;
		\item[(xiii)] a \emph{strongly quasi Baire space} if, for every denumerable family $\mathcal{U}$ of dense open subsets of $\mathbf{X}$, the set $\bigcap\mathcal{U}$ is non-empty and open in $\mathbf{X}$.
	\end{enumerate}
\end{definition} 

Notions (i)--(x) of Definition \ref{s1d10} are not new. The concept of a Loeb space, having roots in \cite{loeb} and \cite{br}, has been investigated, for instance, in \cite{kw4} recently.  Notions (xi)--(xiii) of Definition \ref{s1d10} are new. 
\begin{remark}
	It is known that $\mathbb{R}$ and, for instance, all compact Hausdorff second-countable spaces are Loeb (see, e.g., \cite[Theorem 2.3 and Corollary 2.4]{kw4}).
\end{remark}

Every non-empty Baire space is a quasi Baire space and every strongly Baire space is a Baire space. The following example shows that a quasi Baire space need not be a Baire space.

\begin{example}
	\label{s1e12}
	$[\mathbf{ZF}]$ Let $X=(\mathbb{Q}\cap(0,1))\cup[2, 3]$. Since every separable completely metrizable space is Baire in $\mathbf{ZF}$ (see, e.g., \cite[Theorem 4.102]{her}), the interval $[2,3]$ is a Baire space. This implies that the subspace $\mathbf{X}$ of $\mathbb{R}$ is a quasi Baire space. The family $\mathcal{D}= \{X\setminus\{q\}: q\in \mathbb{Q}\cap(0,1)\}$ is denumerable and every member of $\mathcal{D}$ is a dense open set in $\mathbf{X}$. However, the set $\bigcap\mathcal{D}=[2,3]$ is not dense in $\mathbf{X}$. Therefore, $\mathbf{X}$ is a quasi Baire space which is not Baire.
\end{example}

\begin{definition}
	\label{s1d13}
	If $\mathcal{B}$ is a base of a topological space $\mathbf{X}$, then we say that:
	\begin{enumerate}
		\item[(i)] $\mathcal{B}$ is \emph{strongly stable under countable intersections} if $\emptyset\notin\mathcal{B}$ and, for every countable centered subfamily $\mathcal{U}$ of $\mathcal{B}$, $\bigcap\mathcal{U}\in\mathcal{B}$;
		\item[(ii)] $\mathcal{B}$ is \emph{almost stable under countable intersections} if $\mathcal{B}\cup\{\emptyset\}$ is stable under countable intersections;
		\item[(iii)] $\mathcal{B}$ is \emph{weakly stable under countable intersections} if $\mathcal{B}_{\delta}\subseteq \tau$.
	\end{enumerate}
\end{definition}

\begin{proposition}
	\label{s1p14}
	$[\mathbf{ZF}]$ Suppose that $\mathcal{B}$ is a base of a topological space $\mathbf{X}=\langle X, \tau\rangle$. 
	\begin{enumerate}
		\item[(i)] If $\mathcal{B}$ is strongly stable under countable intersections, then it is almost stable under countable intersections. 
		\item[(ii)] If $\mathcal{B}$ is stable under finite intersections and strongly stable under countable intersections, then $\mathcal{B}$ is stable under countable intersections.
		\item[(iii)] If $\mathcal{B}$ is almost stable under countable intersections, then it is weakly stable under countable intersections.
	\end{enumerate}
\end{proposition}
\begin{proof}
	Let $\mathcal{U}$ be a non-empty countable subfamily of $\mathcal{B}$ and let $U=\bigcap\mathcal{U}$. 
	
	(i) Suppose that $\mathcal{B}$ is strongly stable under countable intersections. If $U\neq\emptyset$, then $\mathcal{U}$ is centered, so $U\in\mathcal{B}$. This implies that $\mathcal{B}_{\delta}\subseteq\mathcal{B}\cup\{\emptyset\}$, so (i) holds.
	
	(ii) Suppose that $\mathcal{B}$ is both strongly stable under countable intersections and stable under finite intersections. Then, since $\mathcal{U}$ is countable, $\mathcal{B}$ is stable under finite intersections and $\emptyset\notin\mathcal{B}$, it follows that the family $\mathcal{U}$ is centered. Hence $U\in\mathcal{B}$ because $\mathcal{B}$ is strongly stable under countable intersetions. This shows that (ii) holds.
	
	It is obvious that (iii) holds.
\end{proof}

\begin{example}
	\label{s1e16}
	$[\mathbf{ZF}]$ (i) The family $\mathcal{B}=\{\{x\}: x\in\mathbb{R}\}\cup\{[x, +\infty): x\in\mathbb{Q}\}$ is a weakly stable under countable intersections base of $\mathbb{R}_{disc}$. That $\mathcal{B}$ is not almost stable under countable intersections follows from the fact that if $(x_n)_{n\in\mathbb{N}}$ is an increasing sequence of rational numbers such that $\lim\limits_{n\to +\infty}x_n=\pi$, then $\bigcap\limits_{n\in\mathbb{N}}[x_n, +\infty)=[\pi, +\infty)\notin\mathcal{B}$. 
	
	(ii) If $\mathbb{N}$ is endowed with the discrete topology, then the family $\mathcal{B}=\{U\subseteq\mathbb{N}: 0<|U|<\aleph_0\vee|\mathbb{N}\setminus U|<\aleph_0\}$ is a base of $\mathbb{N}$ such that $\mathcal{B}$ is almost stable under countable intersections but not strongly stable under countable intersections as $\mathcal{U}=\{\mathbb{N}\setminus\{n\}: n\in\mathbb{N}\}$ is a centered family of members of $\mathcal{B}$ such that $\bigcap\mathcal{U}=\emptyset\notin\mathcal{B}$. 
\end{example}

\begin{definition}
	\label{s1d17}
	For an infinite set $X$ and an element $\infty\notin X$, let $X(\infty)=X\cup\{\infty\}$. Then $\mathbf{X}(\infty)$ is the one-point Hausdorff compactification of the discrete space $X_{disc}$ such that $\infty$ is the unique accumulation point of $\mathbf{X}(\infty)$. More precisely, $\mathbf{X}(\infty)=\langle X(\infty), \tau\rangle$ where $\tau=\mathcal{P}(X)\cup\{X(\infty)\setminus F: F\in [X]^{<\omega}\}$.
\end{definition}

\begin{definition}
	\label{s1d18}
	(Cf. \cite{kw3}.) A topological space $\langle X, \tau\rangle$ is called a \emph{cuf space} if $X$ is a cuf set.
\end{definition}

\begin{definition}
	\label{s1d19}
	For a topological space $\mathbf{X}=\langle X, \tau\rangle$, let $\tau(\mathcal{G}_{\delta}(\mathbf{X}))$ be the topology on $X$ such that $\mathcal{G}_{\delta}(\mathbf{X})$ is a base of $\langle X, \tau(\mathcal{G}_{\delta}(\mathbf{X})\rangle$. Then $(\mathbf{X})_{\delta}=\langle X, \tau(\mathcal{G}_{\delta}(\mathbf{X})\rangle$. 
\end{definition}

Other topological notions, if not defined in this article, are standard and can be found in \cite{en} or \cite{w}.

\subsection{The list of new topological forms}
\label{s1.4}
In the following definition, we introduce several new forms, investigated deeper in Section \ref{s5}.

\begin{definition}
	\label{s1d20}
	\begin{enumerate}
		\item $\mathbf{G}(T_{3\frac{1}{2}}, {\mathcal{G}_{\delta}}_{\delta},\mathcal{G}_\delta)$: For every $T_{3\frac{1}{2}}$-space, the family $\mathcal{G}_{\delta}(\mathbf{X})$ is stable under countable intersections.
		\item $\mathbf{Z}(CR,\mathcal{Z}_{\delta}, \mathcal{Z})$: For every completely regular space $\mathbf{X}$, the family $\mathcal{Z}(\mathbf{X})$ is stable under countable intersections.
		\item $\mathbf{SBSqB}$: Every topological space with a strongly stable under countable intersections base is strongly quasi Baire.
		\item $\mathbf{SB0SqB}$: Every zero-dimensional $T_1$-space having a strongly stable under countable intersections base is strongly quasi Baire.
		\item $\mathbf{ABP}$: Every topological space having an almost stable under countable intersections base is a $P$-space.
		\item $\mathbf{AB0P}$: Every zero-dimensional $T_1$-space having an almost stable under countable intersections base is a $P$-space.
		\item $\mathbf{WBP}$: Every topological space having a weakly stable under countable intersections base is a $P$-space.
		\item $\mathbf{WB0P}$: Every zero-dimensional $T_1$-space having a weakly stable under countable intersections base is a $P$-space.
		\item $\mathbf{CRZOP}$: Every completely regular space whose every zero-set is open is a $P$-space.
		\item $\mathbf{PPP}$: For every pair $\mathbf{X}$, $\mathbf{Y}$ of $P$-spaces, the product $\mathbf{X}\times \mathbf{Y}$ is a $P$-space. 
	\end{enumerate}
\end{definition}

\subsection{The content of the article in brief}
\label{s1.5}
In Section \ref{s2}, for an arbitrary topological space $\mathbf{X}$, we investigate the sets $U_{\aleph_0}(\mathbf{X})$ and $A(\mathbf{X})$ (see Definition \ref{s1d10}) to apply them to strongly zero-dimensional and $P$-spaces in Sections \ref{s3} and \ref{s4}.  Using $U_{\aleph_0}(\mathbf{X})$, we show some equivalences of $\mathbf{CMC}$ and $\mathbf{CAC}$ in Theorem \ref{s2t6}. We notice that $\mathbf{A}(\sigma\sigma)$ and $\mathbf{A}(\delta\delta)$ are equivalent, and prove that the implications $\mathbf{CMC}\rightarrow\mathbf{A}(\sigma\sigma)\rightarrow\mathbf{CMC}_{\omega}$ are true in $\mathbf{ZF}$ (see Theorem \ref{s2t9}). We apply this result to the proof of Theorem \ref{s2t10} which is the main theorem on function spaces in Section \ref{s2}.  Among other facts, Theorem \ref{s2t10} asserts that it holds $\mathbf{ZF+CMC}$ that, for every topological space $\mathbf{X}$, the closure in $C_{u}(\mathbf{X})$ of the set $C(X, \mathbb{R}_{disc})$ is equal to both $U_{\aleph_0}(\mathbf{X})$ and $A(\mathbf{X})$.

In Section \ref{s3}, we prove that, for every topological space $\mathbf{X}$, it holds in $\mathbf{ZF+CMC}$ that $\mathbf{X}$ is strongly zero-dimensional if and only if $C(\mathbf{X})=A(\mathbf{X})=U_{\aleph_0}(\mathbf{X})$ (see Propositions \ref{s3p1} and \ref{s3p2}). We also show in Theorem \ref{s3t5} that $\mathbf{CAC}(\mathbb{R})$ is equivalent to each of the following statements: ``Every zero-dimensional subspace $\mathbf{X}$ of $\mathbb{R}$ has a countable base consisting of clopen sets in $\mathbf{X}$'' and ``Every zero-dimensional subspace of $\mathbb{R}$ is strongly zero-dimensional''.

In Section \ref{s4}, we summarize in Proposition \ref{s4p1} several elementary properties of $P$-spaces in $\mathbf{ZF}$.  Furthermore, we prove in $\mathbf{ZF}$ that $\mathbf{IQDI}$ is equivalent to the sentence ``Every $T_1$-space which is a limit point compact $P$-space is finite'', and the sentence ``Every strongly zero-dimensional $T_1$-space whose underlying set is quasi Dedekind-finite is a $P$-space'' implies $\mathbf{IDI}(\mathbb{R})$ (see Theorem \ref{s4t3}). We also show in Corollary \ref{s4c8} that $\mathbf{IQDI}$ is equivalent to each of the sentences: ``Every infinite space with the co-finite topology is not a $P$-space'' and ``For every infinite discrete space $\mathbf{X}$, the Alexandroff compactification of $\mathbf{X}$ is not a $P$-space''. We obtain several other consequences of $\mathbf{CMC}$ relevant to $P$-spaces. In particular,  Theorem \ref{s4t9} shows that, for every topological space $\mathbf{X}$, it holds in $\mathbf{ZF+CMC}$ that $(\mathbf{X})_{\delta}$ is a $P$-space, and the families $\mathcal{CO}_{\delta}(\mathbf{X})$, $\mathcal{Z}(\mathbf{X})$ and $\mathcal{G}_{\delta}(\mathbf{X})$ are stable under countable intersections.  Theorem \ref{s4t12}(iv) asserts that it holds in $\mathbf{ZF+CMC}$ that a zero-dimensional space $\mathbf{X}$ is a $P$-space if and only if $U_{\aleph_0}(\mathbf{X})=C(\mathbf{X}, \mathbb{R}_{disc})=C(\mathbf{X})$. 

Theorem \ref{s4t10} and its proof show that, among other facts, it holds in $\mathbf{ZF+\neg CMC_{\omega}}$ that there exists a $T_{3\frac{1}{2}}$-space $\mathbf{X}$ for which $(\mathbf{X})_{\delta}$ is not a $P$-space, the families $\mathcal{CO}_{\delta}(\mathbf{X})$, $\mathcal{Z}(\mathbf{X})$ and $\mathcal{G}_{\delta}(\mathbf{X})$ are not stable under countable intersections, and a countable product of compact, countable $T_{3\frac{1}{2}}$-spaces need not be first-countable. 

In Remark \ref{s4r13}, we notice that there is a model of $\mathbf{ZF}$ in which there exists a regular $P$-space which is not zero-dimensional.  Theorem \ref{s4t14} shows that it holds in $\mathbf{ZF+WDC}$ that every regular $P$-space is strongly zero-dimensional.  Theorem \ref{s4t15} concerns closed projections onto $P$-spaces and is a  modified, deeper, extended version of \cite[Theorem 2.1]{mis}. In particular, condition (iii) of Theorem \ref{s4t15} asserts that, in $\mathbf{ZF}$, given a second-countable but not countably compact space $\mathbf{Y}$ and a topological space $\mathbf{X}$, it holds that  $\mathbf{X}$ is a $P$-space if and only if the projection $\pi_X: \mathbf{X}\times\mathbf{Y}\to\mathbf{X}$ is closed. Moreover, the proof of condition (iv) of Theorem \ref{s4t15} shows that, in $\mathbf{ZF}$, a topological space $\mathbf{X}$ is a $P$-space if and only if the projection $\pi_X: \mathbf{X}\times\langle\omega,\omega+1\rangle\to\mathbf{X}$ is closed.  

Theorem \ref{s5t1} shows that, in $\mathbf{ZF+CAC}(\mathbb{R})$, for every subspace $\mathbf{X}$ of $\mathbb{R}$, the family $\mathcal{G}_{\delta}(\mathbf{X})$ is stable under countable intersections; moreover, it holds in $\mathbf{ZF}$ that if $\mathcal{G}_{\delta}(\mathbb{R})$ is stable under countable intersections, then $\mathbf{CAC}_{\omega}(\mathbb{R})$ is true. Other results of Section \ref{s5} concern mainly the forms given in Definition \ref{s1d20} and \ref{s1d7}. Let us mention some of the results. Theorem \ref{s5t3} shows that the following implications hold in $\mathbf{ZF}$: $\mathbf{CMC}\rightarrow(\mathbf{G}(T_{3\frac{1}{2}},\mathcal{G}_{\delta\delta},\mathcal{G}_{\delta})\wedge\mathbf{Z}(CR, \mathcal{Z}_{\delta},\mathcal{Z}))$ and $(\mathbf{G}(T_{3\frac{1}{2}},\mathcal{G}_{\delta\delta}, \mathcal{G}_{\delta})\vee\mathbf{Z}(CR, \mathcal{Z}_{\delta},\mathcal{Z}))\rightarrow (\mathbf{UT}(\aleph_0, cuc, cuc)\wedge\mathbf{UT}(\aleph_0, cuf, cuf))$. Theorem \ref{s5t9}(ii)-(iii) shows that, in $\mathbf{ZF}$,  $\mathbf{CACCLO}$ and $\mathbf{PCACCLO}$ are equivalent, follow from $\mathbf{AB0P}$ and imply $\mathbf{IDI}(\mathbb{R})$. Theorem \ref{s5t9}(i) shows that $\mathbf{ABP}$ holds in $\mathbf{ZF+CMC}(\aleph_0, \infty)$, and $\mathbf{PPP}$ holds in $\mathbf{ZF+ABP}$. That $\mathbf{DC}\rightarrow\mathbf{SBSqB}\rightarrow\mathbf{SB0SqB}\rightarrow\mathbf{CAC}$ and $\mathbf{WB0P}\rightarrow (\mathbf{CAC}(\mathbb{R})\wedge\mathbf{vDCP}(\omega))$ are true in $\mathbf{ZF}$ is shown by Theorem \ref{s5t10}. Corollary \ref{s5c10} states that $\mathbf{SB0SqB}$ implies $\mathbf{ABP}$ but this implication is not reversible in $\mathbf{ZF}$. We are unable to answer the following question:

\begin{question}
	\label{s1q21}
	Is $\mathbf{PPP}$ provable in $\mathbf{ZF}$?
\end{question}

However, by Theorem $\ref{s5t9}$(i), we know that $\mathbf{CMC}(\aleph_0, \infty)$ implies $\mathbf{PPP}$ in $\mathbf{ZF}$, and, by Theorem \ref{s5t12}, it holds in $\mathbf{ZF}$ that if a $P$-space $\mathbf{X}$ is either hereditarily Lindel\"of or a locally compact $T_2$-space, or every point of $\mathbf{X}$ has a well-orderable base of neighborhoods, then, for every $P$-space $\mathbf{Y}$, the space $\mathbf{X}\times\mathbf{Y}$ is a $P$-space.

Section \ref{s6} contains a shortlist of open problems.

\section{Special sets of continuous real functions}
\label{s2}

In this section, we assume that $\mathbf{X}=\langle X, \tau\rangle$ is a given topological space, $\mathbf{Y}=\langle Y, \rho\rangle$ a given metric space, and $E$ is a non-empty set. To have a deeper look at $U_{\aleph_0}(\mathbf{X})$ and $A(\mathbf{X})$, we are concerned with some special subsets of $C(\mathbf{X}, \mathbf{Y})$ that are defined as follows:

\begin{definition}
	\label{s2d1}
	\begin{enumerate}
		\item[(i)] $U_{|E|}(\mathbf{X},\mathbf{Y})$ is the set of all $f\in C(\mathbf{X}, \mathbf{Y})$ such that, for each $\varepsilon>0$, there exists a cover $\mathcal{A}=\{A_j: j\in J\}$ of $X$ such that $|J|\leq |E|$ and, for every $j\in J$, $A_j\in\mathcal{CO}(\mathbf{X})$ and  $\text{osc}_{A_j}(f)\leq\varepsilon$. 
		\item[(ii)] $U_{|E|,AC}(\mathbf{X}, \mathbf{Y})$ (respectively, $U_{|E|,MC}(\mathbf{X}, \mathbf{Y})$) is the set of all $f\in C(\mathbf{X}, \mathbf{Y})$ such that, for each $\varepsilon>0$, there exists a cover $\mathcal{A}=\{A_j: j\in J\}$ of $X$ such that $|J|\leq |E|$, $\mathcal{A}$  is pairwise disjoint and has a choice function (respectively, a multiple choice function), for every $j\in J$, $A_j\in\mathcal{CO}(\mathbf{X})$ and  $\text{osc}_{A_j}(f)\leq\varepsilon$.
		\item[(iii)] $U_{|E|,lf}(\mathbf{X},\mathbf{Y})$ is the set of all functions  $f\in C(\mathbf{X}, \mathbf{Y})$ such that, for every $\varepsilon>0$, there exists a locally finite clopen cover $\{A_j: j\in J\}$ of $X$ such that $|J|\leq|E|$ and, for every $j\in J$, $\text{osc}_{A_j}(f)\leq\varepsilon$.
	\end{enumerate}
	
	If $\mathbf{Y}=\mathbb{R}$, we put $U_{|E|}(\mathbf{X})=U_{|E|}(\mathbf{X}, \mathbb{R})$, $U_{|E|,AC}(\mathbf{X})=U_{|E|, AC}(\mathbf{X}, \mathbb{R})$, $U_{|E|,lf}(\mathbf{X})=U_{|E|,lf}(\mathbf{X},\mathbb{R})$. If $E=\omega$, instead of $|E|$, we write $\aleph_0$. 
\end{definition}

\begin{proposition}
	\label{s2p2}
	$[\mathbf{ZF}]$
	\begin{enumerate}
		\item[(i)]  $U_{|E|,AC}(\mathbf{X},\mathbf{Y})\subseteq U_{|E|,MC}(\mathbf{X},\mathbf{Y})\subseteq U_{|E|,lf}(\mathbf{X}, \mathbf{Y})\subseteq  U_{|E|}(\mathbf{X},\mathbf{Y})$.
		\item[(ii)] the sets  $U_{|E|}(\mathbf{X},\mathbf{Y})$, $U_{|E|,AC}(\mathbf{X},\mathbf{Y})$, $U_{|E|,MC}(\mathbf{X},\mathbf{Y})$ and $U_{|E|,lf}(\mathbf{X}, \mathbf{Y})$  are all closed in $C_u(\mathbf{X},\mathbf{Y})$.
		\item[(iii)] If $E$ is finite, then $U_{|E|,AC}(\mathbf{X},\mathbf{Y})= U_{|E|}(\mathbf{X},\mathbf{Y})$.
	\end{enumerate}
\end{proposition}
\begin{proof}
	It is trivial that (i) and (iii) hold. To prove that $U_{|E|}(\mathbf{X},\mathbf{Y})$ is closed in $C_u(\mathbf{X},\mathbf{Y})$, we consider any function $f\in \cl_{C_u(\mathbf{X}, \mathbf{Y})}(U_{|E|}(\mathbf{X}, \mathbf{Y}))$ and  $\varepsilon\in (0,1]$. There exists $g\in U_{|E|}(\mathbf{X}, \mathbf{Y})$ such that $d_u(f,g)<\frac{\varepsilon}{3}$. There exists a clopen cover $\mathcal{A}=\{A_j: j\in J\}$ of $\mathbf{X}$ such that $|J|\leq |E|$ and, for every $j\in J$, $\text{osc}_{A_j}(g)<\frac{\varepsilon}{3}$. Let $j\in J$ and $x,z\in A_j$. Then $\rho(f(x), f(z))\leq \rho(f(x), g(x))+\rho(g(x),g(z))+\rho(g(z),f(z))\leq \varepsilon$. Hence $f\in U_{|E|}(\mathbf{X}, \mathbf{Y})$. This shows that  $U_{|E|}(\mathbf{X},\mathbf{Y})$ is closed in $C_u(\mathbf{X},\mathbf{Y})$. Using similar arguments, one can check that the sets $U_{|E|,AC}(\mathbf{X}, \mathbf{Y})$, $U_{|E|,MC}(\mathbf{X}, \mathbf{Y})$ and  $U_{|E|,lf}(\mathbf{X}, \mathbf{Y})$ are also closed in $C_u(\mathbf{X},\mathbf{Y})$.
\end{proof}

\begin{proposition}
	\label{s2p3}
	$[\mathbf{ZF}]$  Assume that the space $\mathbf{Y}$ is such that every open cover of $\mathbf{Y}$ has a subcover of size $\leq|E|$. Then  $C(\mathbf{X}, Y_{disc})\subseteq U_{|E|}(\mathbf{X},\mathbf{Y})$. 
\end{proposition}
\begin{proof}
	To show that $C(\mathbf{X}, Y_{disc})$ is a subset of $U_{|E|}(\mathbf{X},\mathbf{Y})$, let us fix $\varepsilon\in (0,1]$. There exists an open cover $\mathcal{V}$ of $\mathbf{Y}$ such that, for every $V\in\mathcal{V}$, $\delta_{\rho}(V)\leq \varepsilon$ and $|\mathcal{V}|\leq|E|$. Let $f\in C(\mathbf{X}, Y_{disc})$. Then, for every $V\in\mathcal{V}$, the set $f^{-1}[V]$ is clopen in $\mathbf{X}$, and the family $\mathcal{A}=\{ f^{-1}[V]: V\in\mathcal{V}\}$ is a cover of $\mathbf{X}$ such that, for every $V\in\mathcal{V}$, $\text{osc}_{f^{-1}[V]}(f)\leq\varepsilon$. Hence $f\in U_{|E|}(\mathbf{X},\mathbf{Y})$.  
\end{proof}

\begin{proposition}
	\label{s2p4}
	$[\mathbf{ZF}]$ 
	\begin{enumerate} 
		\item[(i)] $U_{|E|, AC}(\mathbf{X},\mathbf{Y})\subseteq\cl_{C_u(\mathbf{X}, \mathbf{Y})}(C(\mathbf{X}, Y_{disc})).$
		\item[(ii)] $U_{|E|, MC}(\mathbf{X})\subseteq\cl_{C_u(\mathbf{X})}(C(\mathbf{X}, \mathbb{R}_{disc})).$
	\end{enumerate}
\end{proposition}
\begin{proof}
	Let $f\in U_{|E|, AC}(\mathbf{X}, \mathbf{Y})$ (respectively, $f\in U_{|E|, MC}(\mathbf{X})$) and let $\varepsilon\in (0,1]$. There exists a clopen pairwise disjoint cover $\mathcal{A}=\{A_j: j\in J\}$ of $\mathbf{X}$ such that the family $\mathcal{A}$ has a choice function (respectively, a multiple choice function), $|J|\leq |E|$ and, for every $j\in J$, $\text{osc}_{A_j}(f)\leq \varepsilon$.  Let $a\in\prod\limits_{j\in J}A_j$ (respectively, $a\in\prod\limits_{j\in J}([A_j]^{<\omega}\setminus\{\emptyset\})$. We define a function $h: X\to Y$ by putting, for every $j\in J$, $h(x)=f(a(j))$ (respectively, $h(x)=\min f[a(j)]$) if $x\in A_j$. Then $h\in C(\mathbf{X}, Y_{disc})$ (respectively, $h\in C(\mathbf{X}, \mathbb{R}_{disc})$). To show that $d_u(f, h)\leq\varepsilon$, we fix $x\in X$ and choose $j\in J$ such that $x\in A_j$. Then $\rho(f(x), h(x))\leq \text{osc}_{A_j}(f)\leq\varepsilon$ where $\rho=\rho_e$ if $\langle Y\ \rho\rangle=\langle\mathbb{R},\rho_e\rangle$. Hence $d_u(f,h)\leq\varepsilon$. This completes the proof. 
\end{proof}

\begin{proposition}
	\label{s2p5}
	$[\mathbf{ZF}]$
	\begin{enumerate}
		\item[(i)] Suppose that, for every $\varepsilon\in (0,1]$ and $f\in C(\mathbf{X},\mathbf{Y})$ (respectively, $f\in C(\mathbf{X})$), there exists a partition $\mathcal{V}$ of $f[X]$ such that $|\mathcal{V}|\leq |E|$, the family $\{f^{-1}[V]: V\in\mathcal{V}\}$ has a choice function (respectively, a multiple choice function) and, for every $V\in\mathcal{V}$, $\delta_{\rho}(V)\leq\varepsilon$ (respectively, $\delta_{\rho_e}(V)\leq\varepsilon$). Then 
		$$U_{|E|,AC}(\mathbf{X},\mathbf{Y})=\cl_{C_u(\mathbf{X}, \mathbf{Y})}(C(\mathbf{X},Y_{disc}))$$
		(respectively, $U_{|E|, MC}(\mathbf{X})=\cl_{C_u(\mathbf{X})}(\mathbf{X}, \mathbb{R}_{disc})$).

		\item[(ii)] $U_{\aleph_0}(\mathbf{X})=U_{\aleph_0,lf}(\mathbf{X})$. 
		\item[(iii)] $\mathbf{CMC}$ implies  
		$$U_{\aleph_0, MC}(\mathbf{X})=U_{\aleph_0}(\mathbf{X})=U_{\aleph_0, lf}(\mathbf{X})=\cl_{C_u(\mathbf{X})}(C(\mathbf{X},\mathbb{R}_{disc})).$$
		\item[(iv)] $\mathbf{CAC}$ implies  
		$$U_{\aleph_0, AC}(\mathbf{X})=U_{\aleph_0, MC}(\mathbf{X})=U_{\aleph_0}(\mathbf{X})=U_{\aleph_0, lf}(\mathbf{X})=\cl_{C_u(\mathbf{X})}(C(\mathbf{X},\mathbb{R}_{disc})).$$
	\end{enumerate}
\end{proposition}
\begin{proof}
	(i)  Let $f\in C(\mathbf{X}, Y_{disc})$ (respectively, $f\in C(\mathbf{X}, \mathbb{R}_{disc})$).  Suppose, for an arbitrary $\varepsilon\in(0,1]$, we are given a partition $\mathcal{V}$ of $f[X]$ such that $|\mathcal{V}|\leq |E|$, the family $\mathcal{A}=\{f^{-1}[V]: V\in\mathcal{V}\}$ has a choice function (respectively, a multiple choice function)  and, for every $V\in\mathcal{V}$, $\delta_{\rho}(V)\leq\varepsilon$ (respectively, $\delta_{\rho_e}(V)\leq\varepsilon$). Since every member of $\mathcal{A}$ is clopen in $\mathbf{X}$, we infer that $f\in U_{|E|, AC}(\mathbf{X}, \mathbf{Y})$ (respectively, $f\in U_{|E|, MC}(\mathbf{X})$). To complete the proof of (i),  it suffices to apply Propositions \ref{s2p2}(ii) and \ref{s2p4}.
	
	(ii) It is obvious that $U_{\aleph_0,lf}(\mathbf{X})\subseteq U_{\aleph_0}(\mathbf{X})$. That $U_{\aleph_0}(\mathbf{X})\subseteq U_{\aleph_0,lf}(\mathbf{X})$ follows from the fact that every countable clopen cover of $\mathbf{X}$ has a countable pairwise disjoint clopen refinement. Hence (ii) holds.
	
	For the proof of (iii), we assume $\mathbf{CMC}$. Then $U_{\aleph_0}(\mathbf{X})= U_{\aleph_0,MC}(\mathbf{X})$ and it follows from (i) that $U_{\aleph_0,MC}(\mathbf{X})=\cl_{C_u(\mathbf{X})}(C(\mathbf{X},\mathbb{R}_{disc}))$. This, together with (ii), proves (iii). The proof of (iv) is similar to that of (iii).
\end{proof}

It is interesting that the following theorem holds:

\begin{theorem}
	\label{s2t6}
	$[\mathbf{ZF}]$ 
	\begin{enumerate}
		\item[(i)] $\mathbf{CMC}$ is equivalent to the statement: For every discrete space $\mathbf{X}$, the inclusion $U_{\aleph_0}(\mathbf{X})\subseteq U_{\aleph_0, MC}(\mathbf{X})$ holds.
		\item[(ii)] $\mathbf{CAC}$ is equivalent to the statement: For every discrete space $\mathbf{X}$, the inclusion $U_{\aleph_0}(\mathbf{X})\subseteq U_{\aleph_0, AC}(\mathbf{X})$ holds.
	\end{enumerate}
\end{theorem}
\begin{proof} In view of Proposition \ref{s2p5}(iii), it suffices to show that if, for every discrete space $\mathbf{X}$, the inclusion $U_{\aleph_0}(\mathbf{X})\subseteq U_{\aleph_0, MC}(\mathbf{X})$ (respectively, $U_{\aleph_0}(\mathbf{X})\subseteq U_{\aleph_0, AC}(\mathbf{X})$) holds, then $\mathbf{CMC}$ (respectively, $\mathbf{CAC}$) is true. To this end we fix a disjoint family $\mathcal{A}=\{A_n: n\in\omega\}$ of non-empty sets. We put $X=\bigcup\mathcal{A}$ and define a function $f: X\to\mathbb{R}$ as follows: for every $n\in\omega$ and every $x\in A_n$, $f(x)=n$. Then $f\in U_{\aleph_0}(X_{disc})$. Suppose that $f\in U_{\aleph_0,MC}(X_{disc})$ (respectively, $f\in U_{\aleph_0, AC}(X_{disc})$). Then there exist a set $J\subseteq\omega$ and a cover $\mathcal{V}=\{V_j: j\in J\}$ of $X$, such that $\mathcal{V}$ has a multiple choice function  (respectively, a choice function) and, for every $j\in J$, $\text{osc}_{V_j}(f)\leq\frac{1}{2}$. Then, for every $j\in J$, there exists a unique $n(j)\in\omega$ such that $V_j\subseteq A_{n(j)}$. Furthermore, for every $n\in\omega$, the set $J(n)=\{j\in J: V_j\subseteq A_n\}$ is non-empty. For every $n\in\omega$, let $j(n)=\min J(n)$. Let $\psi$ be a multiple choice function (respectively, a choice function) of $\mathcal{V}$. Then, for every $n\in\omega$, we can define $\phi(n)=\psi(j(n))$ to obtain a multiple choice function (respectively, a choice function) of $\mathcal{A}$. 
\end{proof}

\begin{proposition} 
	\label{s2p7}
	$[\mathbf{ZF}]$ Let $E$ be an infinite well-orderable set. If $X$ is well-orderable, then  $U_{|E|,lf}(\mathbf{X}, \mathbf{Y})=U_{|E|, AC}(\mathbf{X}, \mathbf{Y})$.
\end{proposition}
\begin{proof}
	In view of Proposition \ref{s2p2}(i), it suffices to check that $U_{|E|,lf}(\mathbf{X}, \mathbf{Y})\subseteq U_{|E|, AC}(\mathbf{X}, \mathbf{Y})$. Suppose that $f\in U_{|E|,lf}(\mathbf{X}, \mathbf{Y})$. We fix a real number $\varepsilon>0$. Let $\mathcal{A}=\{A_j: j\in J\}$ be a locally finite clopen cover of $\mathbf{X}$ such that $|J|\leq|E|$ and, for every $j\in J$, $\text{osc}_{A_j}(f)\leq\varepsilon$. Let us fix a well-ordering $\leq_J$ of $J$. Let $j_0$ be the first element of $\langle J, \leq_J\rangle$. Put $B_{j_0}=A_{j_0}$ and, for every $j\in J\setminus\{j_0\}$, let $B_j=A_j\setminus\bigcup\{A_k: k\in J\text{ and } k<_J j\}$. The collection $\{B_j: j\in J\}$ is a clopen cover of $\mathbf{X}$ such that, for every $j\in J$, $\text{osc}_{B_j}(f)\leq\varepsilon$. Let $J^{\ast}=\{j\in J: B_j\neq\emptyset\}$. Then $|J^{\ast}|\leq |E|$. Assuming that $X$ is well-orderable, to show that $f\in U_{|E|, AC}(\mathbf{X},\mathbf{Y})$, it suffices to notice that $\{B_j: j\in J^{\ast}\}$ has a choice function.  
\end{proof}

\begin{proposition}
	\label{s2p8} 
	$[\mathbf{ZF}]$ Let $E$ be an infinite well-orderable set. If  $X$ is well-orderable, then:
	$$U_{|E|, AC}(\mathbf{X})=U_{|E|,lf}(\mathbf{X})=U_{\aleph_0, AC}(\mathbf{X})=U_{\aleph_0,lf}(\mathbf{X})=\cl_{C_u(\mathbf{X})}(C(\mathbf{X},\mathbb{R}_{disc})).$$
\end{proposition}
\begin{proof}
	This follows directly from Proposition \ref{s2p5} and \ref{s2p7}.
\end{proof}

For some of our forthcoming results,  we need the following theorem:

\begin{theorem}
	\label{s2t9}
	$[\mathbf{ZF}]$
	\begin{enumerate}
		\item[(i)] $\mathbf{A}(\sigma\sigma)\leftrightarrow\mathbf{A}(\delta\delta)$.
		\item[(ii)] $\mathbf{CMC}\rightarrow \mathbf{A}(\sigma\sigma)$.
		\item[(iii)] $\mathbf{A}(\sigma\sigma)\rightarrow\mathbf{CMC}_{\omega}$.
	\end{enumerate}
\end{theorem}

\begin{proof}
	Let $\mathcal{A}$ be a family of subsets of a set $X$ such that $\mathcal{A}$ is stable under finite unions and finite intersections. Let $\{\mathcal{U}_n: n\in\omega\}$ be a family of non-empty countable subfamilies of $\mathcal{A}$. 
	
	(i) To see that (i) holds, we consider the family $\mathcal{A}^{c}=\{X\setminus A: A\in\mathcal{A}\}$ and, for every $n\in\omega$, the family $\mathcal{U}_n^{c}=\{X\setminus U: U\in\mathcal{U}_n\}$. Suppose that $\{V_{n,m}: n,m\in\omega\}$ is a family of members of $\mathcal{A}$ (respectively, $\mathcal{A}^c$) such that, for every $n\in\omega$, $\bigcup\mathcal{U}_n=\bigcup\limits_{m\in\omega}V_{n,m}$ (respectively, $\bigcap\mathcal{U}_n^c=\bigcap\limits_{m\in\omega}V_{n,m}$). Then, for every $n\in\omega$, $\bigcap\mathcal{U}_n^c=\bigcap\limits_{m\in\omega}(X\setminus V_{n,m})$ (respectively, $\bigcup\mathcal{U}_n=\bigcup\limits_{m\in\omega}(X\setminus V_{n,m})$).\medskip
	
	(ii)  Without loss of generality, we may assume that, for every $n\in\omega$, the family $\mathcal{U}_n$ is stable under finite unions. For every $n\in\omega$, let $\mathcal{D}_n$ be the family of all sequences $(V_m)_{m\in\omega}$ of members of $\mathcal{U}_n$ such that $\bigcup\mathcal{U}_n=\bigcup\limits_{m\in\omega}V_m$ and, for every $m\in\omega$,  $V_m\subseteq V_{m+1}$. Clearly, for every $n\in\omega$, $\mathcal{D}_n\neq\emptyset$. Assuming $\mathbf{CMC}$, we fix a family $\{\mathcal{E}_n: n\in\omega\}$ of non-empty finite sets such that, for every $n\in\omega$, $\mathcal{E}_n\subseteq\mathcal{D}_n$. Let us fix $n\in\omega$. For $\mathcal{P}\in\mathcal{E}_n$, let $\mathcal{P}=(V(\mathcal{P},n)_m)_{m\in\omega}$. We define $V_{n,m}=\bigcup\{V(\mathcal{P},n)_m: \mathcal{P}\in\mathcal{E}_n\}$ to see that (ii) holds. 
	
	(iii) Now, we assume $\mathbf{A}(\sigma\sigma)$ and consider any family $\mathcal{H}=\{H_n: n\in\omega\}$ of non-empty countable sets. We describe a way of choosing a family $\{F_n: n\in\omega\}$ such that, for every $n\in\omega$, $F_n$ is a non-empty finite subset of $H_n$. We put $X=\bigcup\mathcal{H}$ and $\mathcal{A}=[X]^{<\omega}$. For every $n\in\omega$, let $\mathcal{U}_n=[H_n]^{<\omega}$. Then, for every $n\in\omega$, the family $\mathcal{U}_n$ is a countable subfamily of $\mathcal{A}$, and $H_n=\bigcup\mathcal{U}_n$. Assuming $\mathcal{A}(\sigma\sigma)$, we can fix a family $\{V_{n,m}: n,m\in\omega\}$ of members of $\mathcal{A}$ such that, foe every $n\in\omega$, $H_n=\bigcup\limits_{m\in\omega}V_{n,m}$. To complete the proof, for every $n\in\omega$, we define $m(n)=\min\{m\in\omega: V_{n,m}\neq\emptyset\}$ and $F_n=V_{n,m(n)}$. 
\end{proof}

\begin{theorem}
	\label{s2t10}
	$[\mathbf{ZF}]$ $\mathbf{CMC}$ implies the following:
	\begin{enumerate}
		\item[(i)] for every $f\in U_{\aleph_0}(\mathbf{X})$, the zero-set $Z(f)$ is a countable intersection of clopen sets of $\mathbf{X}$; 
		\item[(ii)] $A(\mathbf{X})=U_{\aleph_0, MC}(\mathbf{X})=U_{\aleph_0, lf}(\mathbf{X})=U_{\aleph_0}(\mathbf{X})=\cl_{C_u(\mathbf{X})}(C(\mathbf{X},\mathbb{R}_{disc}))$.
	\end{enumerate}
\end{theorem}
\begin{proof} To prove that (i) holds, we assume $\mathbf{CMC}$ and fix $f\in U_{\aleph_0}(\mathbf{X})$. By Proposition \ref{s2p5}(iii), $f\in \cl_{C_u(\mathbf{X})}(C(\mathbf{X},\mathbb{R}_{disc}))$. For every $n\in\mathbb{N}$, let $V_n=\{g\in C(\mathbf{X}, \mathbb{R}_{disc}): (\forall x\in X) |f(x)-g(x)|\leq \frac{1}{n}\}$. By $\mathbf{CMC}$, there exists a family $\{H_n: n\in\mathbb{N}\}$ of non-empty finite sets such that, for every $n\in\mathbb{N}$, $H_n\subseteq V_n$. For every $n\in\mathbb{N}$, we can define a function $f_n\in V_n$ as follows: for every $x\in X$, $f_n(x)=\min\{h(x): h\in H_n\}$. Then, for every $n\in\mathbb{N}$, the set $U_n=\{x\in X: |f_n(x)|\leq \frac{1}{n}\}$ is clopen in $\mathbf{X}$. It is easily seen that $Z(f)=\bigcap_{n\in\mathbb{N}}U_n$. Hence (i) holds.
	
	(ii) In view of Propositions \ref{s2p2}(i) and \ref{s2p5}(iii), it remains to check that $\mathbf{CMC}$ implies that $A(\mathbf{X})=U_{\aleph_0}(\mathbf{X})$. To show that $A(\mathbf{X})\subseteq U_{\aleph_0}(\mathbf{X})$, we fix $f\in A(\mathbf{X})$ and $\varepsilon>0$. There exists a collection $\mathcal{J}=\{J_n: n\in\mathbb{N}\}$ of non-degenerate open intervals of $\mathbb{R}$ such that $\mathcal{J}$ is a cover of $\mathbb{R}$ and, for each $n\in\mathbb{N}$, the interval $J_n$ is of length $\leq \varepsilon$. Since $f\in A(\mathbf{X})$, it follows from $\mathbf{CMC}$ that there exists a family $\{\mathcal{U}_n: n\in\mathbb{N}\}$ of countable families of clopen subsets of $\mathbf{X}$ such that, for every $n\in\mathbb{N}$, $f^{-1}[J_n]=\bigcup\mathcal{U}_n$. By $\mathbf{CMC}$ and Theorem \ref{s2t9} (ii), there exists a family $\{\mathcal{V}_n: n\in\omega\}$ of countable subfamilies of $\mathcal{CO}(\mathbf{X})$ such that the family $\mathcal{V}=\bigcup\limits_{n\in\mathbb{N}}\mathcal{V}_n$ is countable and, for every $n\in\mathbb{N}$, $f^{-1}[J_n]=\bigcup\mathcal{V}_n$.  Clearly, $\mathcal{V}$ is a clopen cover of $\mathbf{X}$. Let $\mathcal{V}=\{V_n: n\in\mathbb{N}\}$. We put $U_1=V_1$ and, for every $n\in\mathbb{N}$ with $n>1$, we put $U_n=V_n\setminus\bigcup\limits_{i=1}^{n+1}V_i$.  Then $\{U_n: n\in\mathbb{N}\}$ is a pairwise disjoint clopen cover of $\mathbf{X}$ such that, for every $n\in\mathbb{N}$, the set $M(n)=\{m\in\mathbb{N}: f(U_n)\subseteq J_m\}$ is non-empty. For $n\in\mathbb{N}$, let $m(n)=\min M(n)$ and $r(n)=\frac{\sup J_{m(n)}-\inf J_{m(n)}}{2}$. We define a function $h: X\to\mathbb{R}$ as follows. For every $n\in\mathbb{N}$ and $x\in U_n$, $h(x)=r(n)$. Then $h\in C(\mathbf{X}, \mathbb{R}_{disc})$ and, for every $x\in X$, $|h(x)-f(x)|\leq\varepsilon$. This implies that $f\in \cl_{C_u(\mathbf{X})}(C(\mathbf{X},\mathbb{R}_{disc}))$. It follows from the equality $U_{\aleph_0}(\mathbf{X})=\cl_{C_u(\mathbf{X})}(C(\mathbf{X},\mathbb{R}_{disc}))$ that $f\in U_{\aleph_0}(\mathbf{X})$. Hence $A(\mathbf{X})\subseteq U_{\aleph_0}(\mathbf{X})$.
	
	To show that $U_{\aleph_0}(\mathbf{X})\subseteq A(\mathbf{X})$, let us consider an arbitrary function $f\in U_{\aleph_0}(\mathbf{X})$ and real numbers $a,b$ such that $a<b$. Let $f_a=(f-a)\vee 0$ and $g_b=(f-b)\wedge 0$. One can easily check that $f_a,g_b\in U_{\aleph_0}(\mathbf{X})$. It follows from (i) that the zero-sets $Z(f_a)$ and $Z(g_b)$ are countable intersections of clopen subsets of $\mathbf{X}$. Since $f^{-1}[(a,b)]=(X
	\setminus Z(f_a))\cap (X\setminus Z(g_b))$, the set $f^{-1}[(a,b)]$ is a countable union of clopen subsets of $\mathbf{X}$. Let $O\in\tau_{nat}$. Then $O=\bigcup\mathcal{W}$ where $\mathcal{W}$ is a countable family of bounded open intervals. Since,  $f^{-1}(O)=\bigcup_{W\in\mathcal{W}}f^{-1}[W]$ and, for every $W\in\mathcal{W}$, the set $f^{-1}[W]$ is a countable union of clopen sets, it follows from $\mathbf{CMC}$ and Theorem \ref{s2t9}(ii) that $f^{-1}[O]$  is a countable union of clopen sets of $\mathbf{X}$. Hence $f\in A(\mathbf{X})$. 
\end{proof}

\begin{remark}
	\label{s2r11}
	Let $\mathbf{X}$ be a topological space. It is obvious that $C(\mathbf{X}, \mathbb{Q}_{disc})\subseteq U_{\aleph_0}(\mathbf{X})$. Let us assume $\mathbf{CMC}$ and show that $\cl_{C_u(\mathbf{X})}(C(\mathbf{X}, \mathbb{Q}_{disc}))= U_{\aleph_0}(\mathbf{X})$. To this aim, we fix $f\in U_{\aleph_o}(\mathbf{X})$ and $\varepsilon>0$. By Theorem \ref{s2t10}, $f\in A(\mathbf{X})$. Now, we follow the part of the proof of Theorem \ref{s2t10} saying that $A(X)\subseteq U_{\aleph_0}(\mathbf{X})$. We choose a sequence $(q_n)_{n\in\mathbb{N}}$ of rational numbers such that, for every $n\in\mathbb{N}$, $|r(n)-q_n|<\frac{\varepsilon}{4}$. We define a function $g: X\to\mathbb{Q}$ as follows. For every $n\in\mathbb{N}$ and $x\in U_n$, $g(x)=q_n$. Then, for every $x\in X$, $|g(x)-f(x)|\leq\varepsilon$. Of course, $g\in C(\mathbf{X},\mathbb{Q}_{disc})$. This shows that $f\in \cl_{C_u(\mathbf{X})}(C(\mathbf{X}, \mathbb{Q}_{disc}))$ and, in consequence, $\cl_{C_u(\mathbf{X})}(C(\mathbf{X}, \mathbb{Q}_{disc}))= U_{\aleph_0}(\mathbf{X})$. One can also note that, in the above comments, $\mathbb{Q}$ can be replaced by any separable dense subspace of $\mathbb{R}$.
\end{remark}

\section{Algebras of real functions on strongly zero-dimensional spaces}
\label{s3}

Clearly, every strongly zero-dimensional space is zero-dimensional and every zero-dimensional space is completely regular. It was shown in \cite{ter} that the statement ``For $T_{3\frac{1}{2}}$-space $\mathbf{X}$, it holds that $\mathbf{X}$ is strongly zero-dimensional if and only if every zero-set in $\mathbf{X}$ is the intersection of a countable family of clopen sets of $\mathbf{X}$'' is true in $\mathbf{ZFC}$. To start a deeper analysis of the above-mentioned statement in $\mathbf{ZF}$, let us prove the following two propositions which will be applied in Section \ref{s4}.

\begin{proposition}
	\label{s3p1}
	$[\mathbf{ZF}]$ $\mathbf{CMC}$ implies that if $\mathbf{X}$ is a strongly zero-dimensional space, then every zero-set in $\mathbf{X}$ is the intersection of a countable family of clopen sets in $\mathbf{X}$ and, moreover, $A(\mathbf{X})=U_{\aleph_0}(\mathbf{X})=C(\mathbf{X})$.
\end{proposition}
\begin{proof}
	We assume $\mathbf{CMC}$ and suppose that $\mathbf{X}$ is a strongly zero-dimensional space. Let $f\in C(\mathbf{X})$ and $Z=Z(f)$. For every $n\in\mathbb{N}$, let 
	$$\mathcal{V}_n=\{ V: V\text{ is clopen in }\mathbf{X}\text{ and } Z\subseteq V\subseteq f^{-1}[(-\frac{1}{n}, \frac{1}{n})]\}.$$
	It follows from $\mathbf{CMC}$ that there exists a collection $\{\mathcal{K}_n: n\in\mathbb{N}\}$ such that, for every $n\in\mathbb{N}$, $\mathcal{K}_n$ is a non-empty finite subset of $\mathcal{V}_n$. For every $n\in\mathbb{N}$, let $V_n=\bigcap\mathcal{K}_n$. Then, for every $n\in\mathbb{N}$, $V_n\in\mathcal{V}_n$. Furthermore, $Z=\bigcap\limits_{n\in\mathbb{N}}V_n$. Now, it is clear that if $g\in C(\mathbf{X})$, then $g\in A(\mathbf{X})$, so $C(\mathbf{X})=A(\mathbf{X})$. Hence, it follows from Theorem  \ref{s2t10}(ii) that $A(\mathbf{X})=U_{\aleph_0}(\mathbf{X})=C(\mathbf{X})$. 
\end{proof}

\begin{proposition}
	\label{s3p2}
	$[\mathbf{ZF}]$ 
	Let $\mathbf{X}=\langle X, \tau\rangle$ be a completely regular space.
	\begin{enumerate}
		\item[(i)] If $A(\mathbf{X})=C(\mathbf{X})$,  then $\mathbf{X}$ is strongly zero-dimensional. 
		\item[(ii)] $\mathbf{CMC}$ implies that the space $\mathbf{X}$ is strongly zero-dimensional if and only if $A(\mathbf{X})=C(\mathbf{X})$.
		\item[(iii)] $\mathbf{CMC}$ implies that the space $\mathbf{X}$ is strongly zero-dimensional if and only if $U_{\aleph_0}(\mathbf{X})=C(\mathbf{X})$.
	\end{enumerate}
\end{proposition}
\begin{proof}
	(i) Suppose that $A(\mathbf{X})=C(\mathbf{X})$. Let $Z_1, Z_2$ be a pair of disjoint zero-sets in $\mathbf{X}$. Take a function $f\in C(\mathbf{X})$ such that $Z_1\subseteq f^{-1}[\{0\}]$ and $Z_2\subseteq f^{-1}[\{1\}]$. Put $U_1=f^{-1}[(-\infty, 1)]$ and $U_2=f^{-1}[(0, +\infty)]$. Since $A(\mathbf{X})=C(\mathbf{X})$, there exists a collection $\{W_n: n\in\mathbb{N}\}$ of clopen sets of $\mathbf{X}$ such that $U_1=\bigcup_{n\in\mathbb{N}}W_{2n}$ and $U_2=\bigcup_{n\in\mathbb{N}}W_{2n-1}$. Let $V_1=W_1$ and, for every $n\in\mathbb{N}$ with $n>1$, let $V_n=W_n\setminus\bigcup_{i=1}^{n-1}W_i$. The sets $V_n$ are all clopen in $\mathbf{X}$ and pairwise disjoint. Let $G=\bigcup_{n\in\mathbb{N}}V_{2n}$ and $H=\bigcup_{n\in\mathbb{N}}V_{2n-1}$. Then $G\cup H=X$ and $G\cap H=\emptyset$, so the sets $G,H$ are both clopen in $\mathbf{X}$. Moreover, $G\subseteq U_1$ and $H\subseteq U_2$. Hence $Z_1\subseteq G\subseteq X\setminus Z_2$. This proves that $\mathbf{X}$ is strongly zero-dimensional.
	
	That (ii) holds,  follows from (i) and Proposition \ref{s3p1}. That (iii) holds, follows (ii) and Theorem \ref{s2t10}(i).
\end{proof}

It is worth to notice that the following proposition holds:
\begin{proposition}
	\label{s3p3}
	$[\mathbf{ZF}]$ Let $\mathbf{X}$ be a zero-dimensional space such that $U_{\aleph_0}(\mathbf{X})$ is separable. Then $\mathbf{X}$ is second-countable.
\end{proposition}
\begin{proof}
	Let $\{g_n: n\in\omega\}$ be a dense subset of $U_{\aleph_0}(\mathbf{X})$. For each $n\in\omega$, let $W_n=\{x\in X: |g_n(x)|<\frac{2}{3}\}$. The family $\mathcal{B}=\{W_n: n\in\omega\}$ consists of open sets of $\mathbf{X}$. To check that $\mathcal{B}$ is a base of $\mathbf{X}$, we consider an arbitrary non-empty clopen subset $W$ of $\mathbf{X}$ and a point $x\in W$. For the characteristic function  $f=\chi_{X\setminus W}$ of $X\setminus W$, we have $f\in U_{\aleph_0}(\mathbf{X})$, so there exists $n_f\in\omega$ such that, for every $t\in X$, $|f(t)-g_{n_f}(t)|<\frac{1}{3}$. Since $f(x)=0$, we have $x\in W_{n_f}$. If $t\in W_{n_f}$, then $|f(t)|\leq\frac{1}{3}+|g_{n_f}(t)|<1$. This implies that $W_{n_f}\subseteq W$. Hence $\mathcal{B}$ is a countable base of $\mathbf{X}$.
\end{proof}

\begin{remark}
	\label{s3r4}
	Let $\mathbf{X}=\langle X,\tau\rangle$ be a second-countable space. Then one can show in $\mathbf{ZF}$ that every disjoint family of open sets of $\mathbf{X}$ is countable. Indeed, let $\mathcal{V}$ be a disjoint subfamily of $\tau$ and let $\mathcal{B}=\{U_n: n\in\omega\}$ be a base of $\mathbf{X}$ such that, for every $n\in\omega$, $U_n\neq\emptyset$.  For every $V\in\mathcal{V}$, we can define $n(V)=\min\{n\in\omega: U_n\subseteq V\}$.  Then the function $\mathcal{V}\ni V\to n(v)\in\omega$ is an injection, so $\mathcal{V}$ is countable.
\end{remark}

The following proposition can be proved by using similar arguments to \cite[proof of Theorem 6.2.7]{en}:

\begin{proposition}
\label{s3p5}
$[\mathbf{ZF}]$ If a topological space $\mathbf{X}$ has a countable base consisting of clopen sets, then $\mathbf{X}$ is strongly zero-dimensional. Every zero-dimensional Lindel\"of space is strongly zero-dimensional.
\end{proposition}

Theorem \ref{s3t5} below shows that the statement ``Every second-countable zero-dimensional space is strongly zero-dimensional'' is independent of $\mathbf{ZF}$. It also shows that the statement ``Every second-countable zero-dimensional $T_1$-space has a countable base consisting of clopen sets'' is unprovable in $\mathbf{ZF}$.

\begin{theorem}
	\label{s3t5}
	$[\mathbf{ZF}]$ The following are equivalent:
	\begin{enumerate}
		\item[(i)] $\mathbf{CAC}(\mathbb{R})$;
		\item [(ii)] every zero-dimensional subspace $\mathbf{X}$ of $\mathbb{R}$ has a countable base consisting of clopen sets of $\mathbf{X}$;
		\item[(iii)] for every $Y\subseteq \mathbb{R}$ such that the subspace $\mathbb{R}\setminus Y$ of $\mathbb{R}$ is zero-dimensional, it holds that the subspace $\mathbf{Y}$ of $\mathbb{R}$ is separable;
		\item[(iv)] every zero-dimesional subspace of $\mathbb{R}$ is strongly zero-dimensional.
	\end{enumerate}
\end{theorem}
\begin{proof}
	(i)$\rightarrow$(ii) Let $X\subseteq\mathbb{R}$ be such that the subspace $\mathbf{X}$ of $\mathbb{R}$ is zero-dimensional. It is known from, e.g., \cite[Theorem 4.54]{her} that $\mathbf{CAC}(\mathbb{R})$ is equivalent to the statement: Every subspace of $\mathbb{R}$ is separable. Assuming $\mathbf{CAC}(\mathbb{R})$, we can fix a countable dense subset $D$ of the subspace $\mathbb{R}\setminus X$ of $\mathbb{R}$. Then the family $\{ (a, b)\cap X: a,b\in D\wedge a<b\}$ is a countable base of $\mathbf{X}$ consisting of clopen sets of $\mathbf{X}$.
	
	(ii)$\rightarrow$(iii) Let $Y\subseteq \mathbb{R}$ and $X=\mathbb{R}\setminus Y$. Suppose that the subspace $\mathbf{X}$ of $\mathbb{R}$ has a base $\mathcal{B}=\{U_n: n\in\omega\}$ such that $\mathcal{B}\subseteq\mathcal{CO}(\mathbf{X})$. We may assume that, for every $n\in\omega$, the set $U_n$ is a non-empty bounded subset of $\mathbb{R}$.  For every $n\in\omega$, let $u_n=\inf(U_n)$. Since the sets $U_n$ are clopen in $\mathbf{X}$, we have that, for every $n\in\omega$,  $u_n\in Y$. Then the set $(\mathbb{Q}\cap Y)\cup\{u_n: n\in\omega\}$ is countable and dense in the subspace $\mathbf{Y}$ of $\mathbb{R}$. Hence (ii) implies (iii).\medskip
	
	To prove that each of (iii) and (iv) implies (i), let us fix a family $\mathcal{A}=\{A_n: n\in\mathbb{Z}\}$ of non-empty subsets of $\mathbb{R}$ where $\mathbb{Z}$ is the set of all integers.  Assuming either (iii) or (iv), we show that $\mathcal{A}$ has a choice function. Without loss of generality, we may assume that, for every $n\in\mathbb{Z}$, $A_n$ is a dense subset of the interval $(n, n+1)$, and $\inter_{\mathbb{R}}(A_n)=\emptyset$. In what follows, let $Y=\bigcup\limits_{n\in\mathbb{Z}} A_n$ and $X=\mathbb{R}\setminus Y$. Then the subspace $\mathbf{X}$ of $\mathbb{R}$ is zero-dimensional. 
	
	(iii)$\rightarrow$(i) By (iii), the subspace $\mathbf{Y}$ of $\mathbb{R}$ is separable, and the separability of $\mathbf{Y}$ implies immediately that $\mathcal{A}$ has a choice function. 
	
	(iv)$\rightarrow$(i) Let us assume (iv). For every $n\in\mathbb{Z}$, let $C_n=[n-\frac{1}{8}, n+\frac{1}{8}]\cap X$, $D_n=[n+\frac{1}{4}, n+\frac{3}{4}]\cap X$, $C=\bigcup\limits_{n\in\mathbb{Z}}C_n$ and $D=\bigcup\limits_{n\in\mathbb{Z}}D_n$. Of course, $C,D\in\mathcal{Z}(\mathbf{X})$ and $C\cap D=\emptyset$.  By (iv), the space $\mathbf{X}$ is strongly zero-dimensional, so there exists a set $H\in\mathcal{CO}(\mathbf{X})$ such that $C\subseteq H\subseteq X\setminus D$. We define a function $h\in\prod\limits_{n\in\mathbb{Z}}A_n$ as follows. For every $n\in\mathbb{Z}$, let $h(n)=\sup([ n, n+\frac{1}{4}]\cap H)$. Since $H$ is clopen in $\mathbf{X}$, we infer that, for every $n\in\mathbb{Z}$, $h(n)\in A_n$. Hence $h$ is a choice function of $\mathcal{A}$.
	
	(i)$\rightarrow$(iv) If (i) holds, then (iv) also holds by (ii) and Proposition \ref{s3p5}.
	
\end{proof} 

\section{Basic facts about $P$-spaces in $\mathbf{ZF}$}
\label{s4}

Let us begin our investigation of $P$-spaces in the absence of $\mathbf{AC}$ with the following proposition in which we summarize obvious facts observed in \cite{gh} in $\mathbf{ZFC}$ to emphasize that they hold in $\mathbf{ZF}$. Several new results are also included in the proposition below.

\begin{proposition}
	\label{s4p1}
	$[\mathbf{ZF}]$
	\begin{enumerate}
		\item (Cf. \cite[Theorem 5.3(2)]{gh}.)  A topological space $\mathbf{X}$ is a $P$-space if and only if every $G_{\delta}$-set of $\mathbf{X}$ is open in $\mathbf{X}$. In particular, if $\mathbf{X}$ is a $P$-space, then every zero-set of $\mathbf{X}$ (so also every $c_{\delta}$-set of $\mathbf{X}$) is open in $\mathbf{X}$. 
		\item (Cf. \cite[Corollary 5.6]{gh}.) Every subspace of a $P$-space is a $P$-space. 
		\item (Cf. \cite[Corollary 5.7]{gh}.) Every direct sum of $P$-spaces is a $P$-space. 
		\item Every completely regular $P$-space is strongly zero-dimensional.
		\item If $\mathbf{X}$ is a $P$-space, then $C(\mathbf{X}, \mathbb{R}_{disc})=C(\mathbf{X})$.
		\item A $P$-space is Baire if and only if it is strongly Baire.
		\item If a $T_1$-space $\mathbf{X}$ contains a cuf set $Y$ which has an accumulation point in $\mathbf{X}$, then $\mathbf{X}$ is not a $P$-space.
		\item If a $T_1$-space $\mathbf{X}$ is a $P$-space, then every cuf subspace of $\mathbf{X}$ is closed and discrete.
		\item Every topological space whose underlying set is weakly Dedekind-finite is a $P$-space.
		\item If $X$ is an infinite quasi Dedekind-finite set then, for any element $\infty\notin X$, the one-point compactification $\mathbf{X}(\infty)$ of $X_{disc}$ is a Tychonoff $P$-space.
		\item If a first-countable $T_1$-space is a $P$-space, it is discrete. In particular, every metrizable $P$-space is discrete.
		\item A topological space $\mathbf{X}=\langle X, \tau\rangle$ is a $P$-space if and only if the base $\mathcal{B}=\tau\setminus\{\emptyset\}$ is almost stable under countable intersections.
		\item If a $T_1$-space $\mathbf{X}$ is a $P$-space, then every countable subspace of $\mathbf{X}$ is discrete and closed in $\mathbf{X}$. In particular, if a separable $T_1$-space is a $P$-space, it is discrete and countable.
	\end{enumerate}
\end{proposition}
\begin{proof}
	It is obvious that (1), (3)--(6) hold in $\mathbf{ZF}$. Let us give a proof that (2) is true in $\mathbf{ZF}$ for our proof must differ from that in \cite{gh}. In what follows, we assume that $\mathbf{X}=\langle X, \tau\rangle$ is a topological space, and $Y\subseteq X$. \medskip
	
	(2)  Suppose that $\mathbf{X}$ is a $P$-space. Let $\{U_n: n\in\omega\}$ be a family of subsets of $Y$ such that, for every $n\in\omega$, the set $U_n$ is open in the subspace $\mathbf{Y}$ of $\mathbf{X}$. For every $n\in\omega$, let $V_n=\bigcup\{ V\in\tau: V\cap Y=U_n\}$ and let $V=\bigcap\limits_{n\in\omega}V_n$. Since $\mathbf{X}$ is a $P$-space, the set $V$ is open in $\mathbf{X}$. Since $Y\cap V=\bigcap\limits_{n\in\omega}U_n$, the set $\bigcap\limits_{n\in\omega}U_n$ is open in $\mathbf{Y}$, so $\mathbf{Y}$ is a $P$-space. \medskip
	
	(7) Now, suppose that $\mathbf{X}$ is a $T_1$-space, and $Y=\bigcup\limits_{n\in\omega}Y_n$ where, for every $n\in\omega$, the set $Y_n$ is finite. Suppose that $x$ is an accumulation point of $Y$ in $\mathbf{X}$. Let $Z=Y\cup\{x\}$. In the subspace $\mathbf{Z}$ of $\mathbf{X}$, we have $\{x\}=\bigcap\limits_{n\in\omega}(Z\setminus Y_n)$. Hence, if the space $\mathbf{X}$ were a $P$-space, then $\mathbf{Z}$ would be a $P$-space, so  the set $\{x\}$ would be open in $\mathbf{Z}$. However, since $x$ is an accumulation point of $\mathbf{Z}$, $\{x\}$ is not open in $\mathbf{Z}$. This implies that $\mathbf{X}$ is not a $P$-space. Hence (7) holds. That (8) holds is a consequence of (7).\medskip
	
	(9) Suppose that $X$ is a weakly Dedekind-finite set. Let $\{U_n: n\in\omega\}$ be a collection of subsets of $X$ such that, for every $n\in\omega$, $U_{n+1}\subseteq U_n$. It follows from the definition of a weakly Dedekind-finite set that there exists $n_0\in\omega$ such that, for every $n\in\omega$ with $n_0\in n$, $U_n= U_{n_0}$. Hence, if all the sets $U_n$ are open in $\mathbf{X}$, the set $\bigcap\limits_{n\in\omega} U_n= U_{n_0}$ is open in $\mathbf{X}$, so $\mathbf{X}$ is a $P$-space, and (9) is true.\medskip
	
	(10) Now, assume that $X$ is an infinite quasi Dedekind-finite set. That $\mathbf{X}(\infty)$ is a Tychonoff space was observed in \cite[Proposition 2.12]{kw1}. To prove that $\mathbf{X}(\infty)$ is a $P$-space, we consider any family $\{V_n: n\in\omega\}$ of open sets of $\mathbf{X}(\infty)$ and put $V=\bigcap\limits_{n\in\omega}V_n$. We may assume that, for every $n\in\omega$, $V_{n+1}\subseteq V_n$. If $V\subseteq X$, then $V$ is open in $\mathbf{X}(\infty)$. Suppose that $\infty\in V$. Then, for every $n\in\omega$, the set $K_n= X\setminus V_n$ is finite and $K_{n}\subseteq K_{n+1}$. Since $X$ is quasi Dedekind-finite, there exists $n_0\in\omega$ such that, for every $n\in\omega$ with $n_0\in n$, $K_n=K_{n_0}$. Then $V=V_{n_0}$, so $V$ is open in $\mathbf{X}(\infty)$. Thus, it follows from (1) that $\mathbf{X}(\infty)$ is a $P$-space.
	
	It follows from (1) that (11) and (12) hold.
	
	To show that (13) holds, let us assume that $D$ is a countable subset of a $T_1$-space $\mathbf{X}$ such that $\mathbf{X}$ is a $P$-space. Suppose that $x\in X\setminus D$. Then the set  $U=\bigcap\{X\setminus\{t\}: t\in D\}$ is open in $\mathbf{X}$ and $x\in U\subseteq X\setminus D$. This shows that $D$ is closed in $\mathbf{X}$. Using similar arguments, one can show that, for every $y\in D$, the set $\{y\}$ is open in the subspace $\mathbf{D}$ of $\mathbf{X}$. Hence $\mathbf{D}$ is discrete.
\end{proof}

\begin{remark}
	\label{s4r2}
	We notice that an infinite quasi Dedekind-finite set exists in every model of $\mathbf{ZF}+\neg\mathbf{NAS}$ (see Remark \ref{s1r5}(i)). For instance, infinite quasi Dedekind-finite sets exist in the model $\mathcal{M}37$ in \cite{hr} because $\mathbf{NAS}$ fails in $\mathcal{M}37$. It was shown in \cite{kw2} that $\mathbf{IQDI}$ fails in Cohen's original model $\mathcal{M}1$ of \cite{hr} although $\mathbf{IWDI}$ holds in $\mathcal{M}1$. 
\end{remark}

Corollary 5.4 of \cite{gh} asserts that, in $\mathbf{ZFC}$, every countably compact Tychonoff $P$-space is finite and every locally countably compact Tychonoff $P$-space is discrete. To show that Corollary 5.4 of \cite{gh} may fail in $\mathbf{ZF}$, we shall apply the following theorem:

\begin{theorem}
	\label{s4t3}
	$[\mathbf{ZF}]$
	\begin{enumerate}
		\item The sentence ``Every compact Tychonoff $P$-space is finite'' implies $\mathbf{IQDI}$. 
		\item The sentence ``Every locally compact Tychonoff $P$-space is discrete'' implies $\mathbf{IQDI}$.
		\item $\mathbf{IQDI}$ is equivalent to the sentence ``Every $P$-space which is a limit point compact $T_1$-space is finite''.
		\item The sentence ``Every strongly zero-dimensional $T_1$-space whose underlying set is quasi Dedekind-finite is a $P$-space'' implies $\mathbf{IDI}(\mathbb{R})$.
	\end{enumerate}
\end{theorem}
\begin{proof} 
	It follows from Proposition \ref{s4p1}(10) that (1), (2) and the statement ``If every limit point compact $T_1$ $P$-space is finite,  then $\mathbf{IQDI}$'' are all true in $\mathbf{ZF}$. To complete the proof of (3), we assume $\mathbf{IQDI}$ and consider any infinite limit point compact $T_1$-space $\mathbf{X}=\langle X, \tau\rangle$. It follows from $\mathbf{IQDI}$ that there exists a family $\{K_n: n\in\omega\}$ of finite subsets of $X$ such that the set $Y=\bigcup\limits_{n\in\omega}K_n$ is infinite. Since $\mathbf{X}$ is limit point compact, the set $Y$ has an accumulation point in $\mathbf{X}$. Clearly, $Y$ is a cuf set. Then,  by Proposition \ref{s4p1}(7), $\mathbf{X}$ is not a $P$-space. Hence (3) holds.
	
	To prove (4), let us suppose that $\mathbf{IDI}(\mathbb{R})$ fails. Let $X$ be an infinite Dedekind-finite subset of $\mathbb{R}$. Then the subspace $\mathbf{X}$ of $\mathbb{R}$ with the natural topology is metrizable but not discrete. In view of Proposition \ref{s4p1}(11), $\mathbf{X}$ is not a $P$-space.  Since the family $[\mathbb{R}]^{<\omega}\setminus\{\emptyset\}$ has a choice function, the set $X$ is quasi Dedekind-finite. The set $\mathbb{Q}\cap X$ is finite, hence, if $U$ is a non-empty open subset of $\mathbb{R}$, there exist sequences $(a_n)_{n\in\omega}$ and $(b_n)_{n\in\omega}$ of numbers from $\mathbb{Q}\setminus X$ such that, for every $n\in\omega$, $a_n<b_n$, and $U=\bigcup\limits_{n\in\omega}(a_n, b_n)$. Then $U\cap X=\bigcup\limits_{n\in\omega}(X\cap(a_n, b_n))$ and, for every $n\in\omega$, the set $X\cap (a_n, b_n)$ is clopen in $\mathbf{X}$. This implies that $A(\mathbf{X})=C(\mathbf{X})$, so, by Proposition \ref{s3p2}(i), the space $\mathbf{X}$ is strongly zero-dimensional. This completes the proof of (4).
\end{proof}

Since Corollary 5.4 of \cite{gh} holds in $\mathbf{ZFC}$ but, in view of Remark \ref{s4r2} and Theorem \ref{s4t3}, it fails in the model $\mathcal{M}37$ of \cite{hr}, we deduce that the following theorem holds:

\begin{theorem}
	\label{s4t4}
	Corollary 5.4 of \cite{gh} is independent of $\mathbf{ZF}$.
\end{theorem}

\begin{remark}
	\label{s4r5}
	Since,  $\mathbf{IDI}(\mathbb{R})$ is false, for instance, in Cohen's original model $\mathcal{M}1$ in \cite{hr}, it follows from Theorem \ref{s4t3}(iv) that the statement  ``Every strongly zero-dimensional $T_1$-space whose underlying set is quasi Dedekind-finite is a $P$-space'' is unprovable in $\mathbf{ZF}$.
\end{remark}

\begin{definition}
	\label{s4d6}
	For a set $X$, let $\tau_{cof}=\{\emptyset\}\cup\{X\setminus F: F\in [X]^{<\omega}\}$. Then $\tau_{cof}$ is the \emph{cofinite topology} on $X$, and $\mathbf{X}_{cof}=\langle X, \tau_{cof}\rangle$.
\end{definition}

\begin{theorem}
	\label{s4t7}
	$[\mathbf{ZF}]$  
	\begin{enumerate}
		\item[(i)] For every non-empty set $X$, it holds that $\mathbf{X}_{cof}$ is a $P$-space if and only if $X$ is quasi Dedekind-finite. 
		\item[(ii)] For every infinite set $X$ and every element $\infty$ with $\infty\notin X$, the space $\mathbf{X}(\infty)$ (i.e., the Alexandroff compactification of $X_{disc}$) is a $P$-space if and only if $X$ is quasi Dedekind-finite.
	\end{enumerate}
\end{theorem}
\begin{proof}
	If $X$ is a finite set,  then $\mathbf{X}_{cof}$ is a $P$-space and $X$ is quasi Dedekind-finite. In what follows, we assume  that $X$ is an infinite set and $\infty$ is an element such that $\infty\notin X$. The spaces $\mathbf{X}_{cof}$ and $\mathbf{X}(\infty)$ are both limit point compact $T_1$-spaces. Therefore, the proof of Theorem \ref{s4t3}(3) shows that if $X$ is quasi Dedekind-infinite, then the spaces $\mathbf{X}_{cof}$ and $\mathbf{X}(\infty)$ are not $P$-spaces. 
	
	Now, let us assume that $X$ is an infinite quasi Dedekind-finite set. It follows from Proposition \ref{s4p1}(10) that $\mathbf{X}(\infty)$ is a $P$-space. To complete the proof, let us show that $\mathbf{X}_{cof}$ is a $P$-space. To this end, we fix a point $x_0\in X$ and a family $\{K_n: n\in\omega\}$  of finite subsets of $X$ such that, for every $n\in\omega$, $K_n\subseteq K_{n+1}$ and  $x_0\in\bigcap\limits_{n\in\omega}(X\setminus K_n)$. Since $X$ is quasi Dedekind-finite, there exists $n_0\in\omega$ such that, for every $n\in\omega$ with $n_0\in n$, $K_n=K_{n_0}$. Then $x_0\in\bigcap\limits_{n\in\omega}(X\setminus K_n)=X\setminus K_{n_0}$, so $x_0$ is a $P$-point of $\mathbf{X}_{cof}$. 
\end{proof}

\begin{corollary}
	\label{s4c8}
	$[\mathbf{ZF}]$ The following are equivalent:
	\begin{enumerate}
		\item[(i)] $\mathbf{IQDI}$;
		\item[(ii)] for every infinite set $X$, the space $\mathbf{X}_{cof}$ is not a $P$-space;
		\item[(iii)] for every infinite set $X$ and any element $\infty\notin X$, the space $\mathbf{X}(\infty)$ is not a $P$-space.
	\end{enumerate}
\end{corollary}

\begin{theorem}
	\label{s4t9}
	$[\mathbf{ZF}]$ $\mathbf{CMC}$ implies the following:
	\begin{enumerate}
		\item[(i)] for every topological space $\mathbf{X}$, the families $\mathcal{CO}_{\delta}(\mathbf{X})$, $\mathcal{Z}(\mathbf{X})$ and $\mathcal{G}_{\delta}(\mathbf{X})$ are stable under countable intersections;
		\item[(ii)] for every topological space $\mathbf{X}$, the space $(\mathbf{X})_{\delta}$ is a $P$-space.
	\end{enumerate}
	Furthermore, (i) implies (ii).
\end{theorem}
\begin{proof}
	Let $\mathbf{X}=\langle X, \tau\rangle$ be a topological space. It is obvious that if $\mathcal{G}_{\delta}(\mathbf{X})$ is stable under countable intersections, then $(\mathbf{X})_{\delta}$ is a $P$-space. Hence (i) implies (ii).
	
	Now, assuming $\mathbf{CMC}$, we prove that (i) holds. First, let us show that $\mathcal{Z}(\mathbf{X})$ is stable under countable intersections. To this aim, we consider any family $\{Z_n: n\in\omega\}$ of members of $\mathcal{Z}(\mathbf{X})$. For every $n\in\omega$, let $\mathcal{F}_n=\{f\in C(\mathbf{X}): Z_n=Z(f)\}$. By $\mathbf{CMC}$, there exists a family $\{\mathcal{K}_n: n\in\omega\}$ such that, for every $n\in\omega$, $\mathcal{K}_n$ is a non-empty finite subset of $\mathcal{F}_n$. For every $n\in\omega$, let $f_n=\sum\limits_{f\in\mathcal{K}_n}(|f|\wedge 1)$. Let $f=\sum\limits_{n\in\omega}\frac{f_n}{2^n}$. Then $f\in C(\mathbf{X})$ and $Z(f)=\bigcap\limits_{n\in\omega}Z_n$. Hence $\bigcap\limits_{n\in\omega}Z_n\in\mathcal{Z}(\mathbf{X})$.
	
	Assuming $\mathbf{CMC}$, let us show that $\mathcal{G}_{\delta}(\mathbf{X})$ and $\mathcal{CO}_{\delta}(\mathbf{X})$ are both stable under countable intersections. To this aim, we consider any family $\{G_n: n\in\omega\}$ of members of $\mathcal{G}_{\delta}(\mathbf{X})$ (respectively, of $\mathcal{CO}_{\delta}(\mathbf{X})$). For every $n\in\omega$, let $\mathcal{G}_n=\{\mathcal{D}\in[\tau]^{\leq\omega}\setminus\{\emptyset\}: G_n=\bigcap\mathcal{D}\}$ (respectively, $\mathcal{G}_n=\{\mathcal{D}\in [\mathcal{CO}]^{\leq\omega}\setminus\{\emptyset\}: G_n=\bigcap\mathcal{D}\}$). By $\mathbf{CMC}$, there exists a family $\{\mathcal{E}_n: n\in\omega\}$ such that, for every $n\in\omega$, $\mathcal{E}_n$ is a non-empty finite subset of $\mathcal{G}_n$. Let us fix $n\in\omega$. Suppose that $k(n)\in\mathbb{N}$ is such that $\mathcal{E}_n=\{\mathcal{D}_i: i\in k(n)\}$. Then we define $\mathcal{U}_n=\{\bigcap D_i: (\forall i\in k(n)) D_i\in\mathcal{D}_i\}$. Clearly, $\mathcal{U}_n$ is countable and $\mathcal{U}_n\subseteq\tau$ (respectively, $\mathcal{U}_n\subseteq\mathcal{CO}(\mathbf{X})$); furthermore, $G_n=\bigcap\mathcal{U}_n$, and $\mathcal{U}_n$ does not depend on our choice of the enumeration of $\mathcal{E}_n$.  In this way, we have defined a family $\{\mathcal{U}_n: n\in\omega\}$ of non-empty countable families of $\tau$ (respectively, of $\mathcal{CO}(\mathbf{X})$). It follows from Theorem \ref{s2t9} that there exists a family $\{V_{n,m}: n,m\in\omega\}$ of open (respectively, clopen) sets of $\mathbf{X}$ such that $\bigcap\limits_{n\in\omega}G_n=\bigcap\limits_{n\in\omega}\bigcap\limits_{m\in\omega}V_{n,m}$. This completes the proof of (i).
\end{proof}

In $\mathbf{ZFC}$, for every topological space $\mathbf{X}$, the space $(\mathbf{X})_{\delta}$ is a $P$-space. That the situation is different in $\mathbf{ZF}$ is shown by the following theorem:

\begin{theorem}
	\label{s4t10}
	$[\mathbf{ZF}]$ Each of the following sentences (i)--(v) implies $\mathbf{CMC}_{\omega}$:
	\begin{enumerate}
		\item[(i)] for every $T_{3\frac{1}{2}}$-space $\mathbf{X}$, at least one of the  families $\mathcal{G}_{\delta}(\mathbf{X})$ and $\mathcal{Z}(\mathbf{X})$ is stable under countable intersections;
		\item[(ii)] for every strongly zero-dimensional $T_1$-space $\mathbf{X}$, the family $\mathcal{CO}_{\delta}(\mathbf{X})$ is stable under countable intersections;
		\item[(iii)] for every $T_{3\frac{1}{2}}$-space $\mathbf{X}$, the space $(\mathbf{X})_{\delta}$ is a $P$-space;
		\item[(iv)] every countable product of countable, compact metrizable spaces is metrizable;
		\item[(v)] every countable product of countable, compact, first-countable $T_{3\frac{1}{2}}$-spaces is first-countable. 
	\end{enumerate}
\end{theorem}
\begin{proof}
	Let $\mathcal{A}=\{A_n: n\in\omega\}$ be a disjoint family of infinite countable sets.  We fix a sequence $(\infty_n)_{n\in\omega}$ of elements such that none of the elements $\infty_n$ is a member of $\bigcup\mathcal{A}$. For every $n\in\omega$, let $X_n=A_n\cup\{\infty_n\}$  and let $\mathbf{X}_n=\mathbf{A}_n(\infty_n)$. We consider the set $X=\prod\limits_{n\in\omega}X_n$ and the space $\mathbf{X}=\prod\limits_{n\in\omega}\mathbf{X}_n$.  For every $n\in\omega$, the space $\mathbf{X}_n$ is countable, compact, metrizable and strongly zero-dimensional. It is obvious that, for every $n\in\omega$, the set $\{\infty_n\}$ is a $c_{\delta}$-set in $\mathbf{X}_n$ and, therefore, $G_n=\pi_n^{-1}(\infty_n)$ is a $c_{\delta}$-set in $\mathbf{X}$. Let $\infty\in\prod\limits_{n\in\omega}X_n$ be defined as follows: for every $n\in\omega$, $\infty(n)=\infty_n$. It is easily seen that $\{\infty\}=\bigcap\limits_{n\in\omega}G_n$. Each of the statements (i)--(v) implies that there exists a family $\{V_j: j\in\omega\}$ of open sets of $\mathbf{X}$ such that  $\{\infty\}=\bigcap\limits_{j\in\omega}V_j$. Now, let us modify and clarify several arguments given in \cite[proof of Theorem 2.2]{kk}. For every pair $j,k$ of elements of $\omega$, we put $V_{j,k}=\{x\in V_j: (\forall n\in\omega\setminus\{k\}) x(n)=\infty_n\}$. Then $\{\infty\}=\bigcap\limits_{j,k\in\omega}V_{j,k}$. Clearly, if $j,k\in\omega$, then $\pi_k[V_{j,k}]$ is an open neighborhood of $\infty_k$ in $\mathbf{X}_k$. Suppose that there exists $k_0\in\omega$ such that, for every $j\in\omega$, $A_{k_0}\subseteq\pi_{k_0}[V_{j, k_0}]$. Fix an element $a\in A_{k_0}$ and define an element $b\in X$ as follows: $b(k_0)=a$ and, for every $n\in\omega\setminus\{k_0\}$, $b(n)=\infty_n$. Then $b\neq\infty$ and $b\in\bigcap\limits_{j,k\in\omega}V_{j,k}$. This is impossible because $\infty$ is the unique element of $\bigcap\limits_{j,k\in\omega}V_{j,k}$. The contradiction obtained proves that, for every $k\in\omega$, the set $C_k=\{j\in\omega: A_j\setminus\pi_k[V_{j.k}]\neq\emptyset\}$ is non-empty. For every $k\in\omega$, we can define $j_k=\min C_k$. Then, for every $k\in\omega$, $A_k\setminus\pi_{k}[V_{j_k, k}]$ is a non-empty finite subset of $A_k$. Hence $\mathcal{A}$ has a multiple choice function. 
\end{proof}

\begin{remark}
	\label{s4r11}
	It is known that $\mathbf{CMC}_{\omega}$ is false, for instance, in Pincus' Model II, labeled as model $\mathcal{M}29$ in \cite{hr}. In consequence, none of the statements (i)-(v) listed in Theorem \ref{s4t9} is provable in $\mathbf{ZF}$.
\end{remark}

\begin{theorem}
	\label{s4t12}
	$[\mathbf{ZF}]$ $\mathbf{CMC}$ implies that, for every topological space $\mathbf{X}$, the following conditions are satisfied:
	\begin{enumerate}
		\item[(i)] if $\mathbf{X}$ is completely regular and every zero-set in $\mathbf{X}$ is open, then $\mathbf{X}$ is a $P$-space;
		\item[(ii)] if $\mathbf{X}$ is a zero-dimensional space such that every member of $\mathcal{CO}_{\delta}(\mathbf{X})$ is open in $\mathbf{X}$, then $\mathbf{X}$ is a $P$-space. 
		\item[(iii)] if $\mathbf{X}$ is completely regular and $C(\mathbf{X},\mathbb{R}_{disc})=C(\mathbf{X})$, then $\mathbf{X}$ is a $P$-space; 
		\item[(iv)]  if $\mathbf{X}$ is zero-dimensional, then $\mathbf{X}$ is a $P$-space if and only if $U_{\aleph_0}(\mathbf{X})=C(\mathbf{X},\mathbb{R}_{disc})=C(\mathbf{X})$.
	\end{enumerate}
\end{theorem}
\begin{proof} Let us assume $\mathbf{CMC}$.
	
	(i) Suppose that $\mathbf{X}$ is a completely regular space whose every zero-set is open in $\mathbf{X}$. Let $\{U_n: n\in\omega\}$ be a family of open sets of $\mathbf{X}$ and let $x\in U=\bigcap\limits_{n\in\omega}U_n$. Since the family $\mathcal{Z}(\mathbf{X})$ is stable under finite intersections, it follows from the complete regularity of $\mathbf{X}$ and from $\mathbf{CMC}$ that we can fix a collection $\{Z_n: n\in\omega\}$ of zero-sets of $\mathbf{X}$ such that, for every $n\in\omega$, $x\in Z_n\subseteq U_n$. Let $Z=\bigcap_{n\in\omega}Z_n$. It follows from Theorem \ref{s4t9}(i) that $Z\in\mathcal{Z}(\mathbf{X})$. By Proposition \ref{s4p1}(1), the set $Z$ is open in $\mathbf{X}$. Since $x\in Z\subseteq U$ and $x$ is an arbitrary point of $U$, we deduce that the set $U$ is open in $\mathbf{X}$. Hence $\mathbf{X}$ is a $P$-space.\medskip
	
	(ii) The proof of (ii) is similar to that of (i), so we omit it.
	
	(iii) Suppose that $\mathbf{X}$ is a topological space such that $C(\mathbf{X}, \mathbb{R}_{disc})=C(\mathbf{X})$. Then every zero-set of $\mathbf{X}$ is open in $\mathbf{X}$, so, by $\mathbf{CMC}$, (iii) follows from (i).\medskip
	
	(iv) Now, suppose that $\mathbf{X}$ is a zero-dimensional $P$-space. By Proposition \ref{s4p1}(5), $C(\mathbf{X}, \mathbb{R}_{disc})=C(\mathbf{X})$. Let us notice that, by Proposition \ref{s4p1}(4), $\mathbf{X}$ is strongly zero-dimensional. Therefore, it follows from Proposition \ref{s3p2}(iii) that $U_{\aleph_0}(\mathbf{X})=C(\mathbf{X})$. On the other hand, if $\mathbf{X}$ is a zero-dimensional space such that $U_{\aleph_0}(\mathbf{X})=C(\mathbf{X},\mathbb{R}_{disc})=C(\mathbf{X})$, then $\mathbf{X}$ is a $P$-space by (iii).
\end{proof}

\begin{remark}
	\label{s4r13}
	We know from \cite[Example 2.4]{gt} that there is a model $\mathcal{M}$ of $\mathbf{ZF}$ in which there exists a compact Suslin line $\mathbf{X}$ on which every continuous real-valued function is constant. So, $\mathbf{X}$ is regular but not completely regular in $\mathcal{M}$. Let us notice that, in $\mathcal{M}$, the Suslin line $\mathbf{X}$ has the property that the only subsets of $X$ are finite unions of sets that are either intervals or singletons of $\mathbf{X}$. This implies that, in $\mathcal{M}$, $\mathbf{X}$ is a $P$-space. However, $\mathbf{X}$ is not zero-dimensional. Therefore the statement ``Every regular $P$-space is zero-dimensional'' is unprovable in $\mathbf{ZF}$. It would be interesting to know more about the set-theoretic status of this statement. 
\end{remark}

\begin{theorem}
	\label{s4t14}
	$[\mathbf{ZF}]$ $\mathbf{WDC}$ implies that every regular $P$-space  is zero-dimensional, so also completely regular. In consequence, $\mathbf{WDC}$ implies that every regular $P$-space is strongly zero-dimensional and a regular $P$-space is zero-dimensional if and only if it is completely regular. 
\end{theorem}
\begin{proof}
	We assume $\mathbf{WDC}$ and consider any regular $P$-space $\mathbf{X}=\langle X, \tau\rangle$. We fix a point $x\in X$ and a set $W\in\tau$ such that $x\in W$. Let $\mathcal{B}(x)=\{U\in\tau: x\in U\subseteq W\}$. We define a binary relation $R$ on $\mathcal{B}(x)$ as follows: $(\forall U,V\in\mathcal{B}(x))(\langle U, V\rangle\in R\leftrightarrow \cl_{\mathbf{X}}(V)\subseteq U)$. By $\mathbf{WDC}$, we can fix a sequence $(\mathcal{F}_n)_{n\in\omega}$ of non-empty finite subfamilies of $\mathcal{B}(x)$ such that: ($\forall n\in\omega)(\forall U\in \mathcal{F}_n)(\exists V\in\mathcal{F}_{n+1})(\langle U, V\rangle\in R)$. For every $n\in\omega$, we define $U_n=\bigcap\mathcal{F}_n$. Then, for every $n\in\omega$, $U_n\in\mathcal{B}(x)$ and $\cl_{\mathbf{X}}(U_{n+1})\subseteq U_n$. Since $\mathbf{X}$ is a $P$-space, the set $U=\bigcap\limits_{n\in\omega}U_n$ is open in $\mathbf{X}$ by Proposition \ref{s4p1}(1). Of course, $x\in U\subseteq W$. Since $U=\bigcap\limits_{n\in\omega}\cl_{\mathbf{X}}(U_{n+1})$, the set $U$ is clopen in $\mathbf{X}$. Hence $\mathbf{X}$ is zero-dimensional, so also completely regular. It follows from Proposition \ref{s4p1}(4) that $\mathbf{X}$ is strongly zero-dimensional. To complete the proof, it suffices to apply Proposition \ref{s4p1}(4) once again.
\end{proof}

Theorem 2.1 of \cite{mis} asserts that a topological space $\mathbf{X}$ is a $P$-space if and only if, for every Lindel\"of space $\mathbf{Y}$, the projection $\pi_X:\mathbf{X}\times\mathbf{Y}\to\mathbf{X}$ is closed; however, the proof of this theorem in \cite{mis} is not a proof in $\mathbf{ZF}$. Therefore, let us modify Theorem 2.1 of \cite{mis} to Theorem \ref{s4t14} below and give a proof of it in $\mathbf{ZF}$.

\begin{theorem}
	\label{s4t15}
	$[\mathbf{ZF}]$ 
	\begin{enumerate}
		\item[(i)] (\cite[Theorem 4]{han}.) If $\mathbf{Y}$ is a topological space which is not countably compact, and  $\mathbf{X}$ is a topological space such that the projection $\pi_{X}:\mathbf{X}\times\mathbf{Y}\to\mathbf{X}$ is closed, then $\mathbf{X}$ is a $P$-space.
		\item[(ii)] If $\mathbf{X}=\langle X, \tau_X\rangle$ and $\mathbf{Y}=\langle Y, \tau_Y\rangle$ are topological spaces such that $\mathbf{X}$ is a $P$-space and $\mathbf{Y}$ is either  Lindel\"of or second-countable, then the projection $\pi_X:\mathbf{X}\times\mathbf{Y}\to\mathbf{X}$ is closed.
		\item[(iii)] If $\mathbf{Y}$ is a fixed second-countable space which is not countably compact, then, for every topological space  $\mathbf{X}$, it holds that $\mathbf{X}$ is a $P$-space if and only if the projection $\pi_{X}:\mathbf{X}\times\mathbf{Y}\to\mathbf{X}$ is closed.
		
		\item[(iv)] There exists a second-countable Lindel\"of $T_0$-space $\mathbf{Y}$ such that, for every topological space $\mathbf{X}$, it holds that $\mathbf{X}$ is a $P$-space if and only if the projection $\pi_X:\mathbf{X}\times\mathbf{Y}\to\mathbf{X}$ is closed. 
		
		\item[(v)] $\mathbf{CAC}(\mathbb{R})$ is equivalent to the statement: For every $P$-space $\mathbf{X}$, there exists a non-compact Lindel\"of $T_1$-space $\mathbf{Y}$ such that the projection $\pi_{X}:\mathbf{X}\times\mathbf{Y}\to\mathbf{X}$ is closed.
	\end{enumerate}
\end{theorem}

\begin{proof}  We notice that (i) is equivalent to Theorem 4 of \cite{han} and the proof of Theorem 4 in \cite{han} is a proof in $\mathbf{ZF}$. Hence (i) holds. 
	
	To prove (ii)--(iv), we fix a topological space $\mathbf{X}=\langle X, \tau_X\rangle$.
	
	(ii) Suppose that $\mathbf{X}$ is a $P$-space and $\mathbf{Y}=\langle Y, \tau_Y\rangle$ is topological space such that $\mathbf{Y}$ is either  Lindel\"of or second-countable. Suppose that the projection $\pi_X:\mathbf{X}\times\mathbf{Y}\to\mathbf{X}$ is not closed. We fix a closed set $C$ in $\mathbf{X}$ such that the set $\pi_X[C]$ is not closed in $\mathbf{X}$. We fix $x\in\cl_{\mathbf{X}}(\pi_X[C])\setminus\pi_X[C]$. We fix a base $\mathcal{B}_Y$ of $\mathbf{Y}$ such that, if $\mathbf{Y}$ is second-contable, then $\mathcal{B}_Y$ is countable; if $\mathbf{Y}$ is Lindel\"of, then $\mathcal{B}_Y=\tau_Y$. Let $\mathcal{V}=\{V\in\mathcal{B}_{Y}: (\exists U\in\tau_X)(x\in U\wedge C\cap (U\times V)=\emptyset)\}$. Since $x\notin\pi_X[C]$ and $C$ is closed in $\mathbf{X}$, it follows that  $\mathcal{V}$ is an open cover of $\mathbf{Y}$. Under our assumptions about $\mathbf{Y}$ and $\mathcal{B}_Y$, there exists a family $\{V_n: n\in\omega\}$ of members of $\mathcal{V}$ such that $Y=\bigcup\limits_{n\in\omega}V_n$. For every $n\in\omega$, let $U_n=\bigcup\{W\in\tau_X: x\in W\wedge C\cap(W\times V_n)=\emptyset\}$. Since $\mathbf{X}$ is a $P$-space, the set $U=\bigcap\limits_{n\in\omega}U_n$ is an open neighborhood of $x$ in $\mathbf{X}$. Since $C\cap(U\times Y)=\emptyset$, we have $U\cap \pi_X[C]=\emptyset$. On the other hand, $U\cap\pi_X[C]\neq\emptyset$ because $x\in\cl_{\mathbf{X}}(\pi_X[C])$. The contradiction obtained shows that $\pi_X$ is closed. Hence (ii) holds.
	
	That (iii) holds, follows from (i) and (ii).
	
	(iv) Let $\tau=\omega+1$. Then $\tau$ is a topology on $\omega$ such that the space $\mathbf{Y}=\langle \omega, \tau\rangle$ is a non-compact, second-countable, Lindel\"of $T_0$-space. Hence (iv) follows from (iii).
	
	(v) Let $\mathbf{X}$ be a $P$-space. Assuming  $\mathbf{CAC}(\mathbb{R})$, we infer that, by \cite[Theorem 4.54]{her}, $\mathbb{R}$ is Lindel\"of and, in virtue of (ii),  $\pi_X: \mathbf{X}\times{\mathbb{R}}\to\mathbf{X}$ is closed. On the other hand, if there exists a non-compact Lindel\"of $T_1$-space, then $\mathbf{CAC}(\mathbb{R})$ holds by \cite[Theorem 7.2]{her}. Hence (iv) holds.
\end{proof}

\begin{corollary}
	\label{s4c16} 
	$[\mathbf{ZF}]$ For every topological space $\mathbf{X}$, the following are equivalent:
	\begin{enumerate}
		\item[(i)] $\mathbf{X}$ is a $P$-space;
		\item[(ii)] the projection $\pi_X:\mathbf{X}\times\mathbb{R}\to\mathbf{X}$ is closed;
		\item[(iii)] the projection $\pi_X: \mathbf{X}\times \mathbb{N}\to\mathbf{X}$ is closed.
	\end{enumerate}
\end{corollary}

The following proposition is trivial:

\begin{proposition}
	\label{s4p17}
	$[\mathbf{ZF}]$ If $\mathbf{X}$ is a topological space whose every singleton is of type $G_{\delta}$ (in particular, if $\mathbf{X}$ is first-countable), then the space $(\mathbf{X})_{\delta}$ is discrete, so, for every topological space $\mathbf{Y}$, the projection $\pi_{\delta}:(\mathbf{X})_{\delta}\times\mathbf{Y}\to (\mathbf{X})_{\delta}$ is closed.
\end{proposition}

\section{The forms of Definition \ref{s1d20}}
\label{s5}

Regarding Theorems \ref{s4t9}(i) and \ref{s4t10}(i), it is worth mentioning that the following theorem also holds:

\begin{theorem}
	\label{s5t1}
	$[\mathbf{ZF}]$
	\begin{enumerate}
		\item[(i)] $\mathbf{CAC}(\mathbb{R})$ implies that, for every subspace $\mathbf{X}$ of $\mathbb{R}$,
		the family $\mathcal{G}_{\delta}(\mathbf{X})$ is stable under countable intersections.
		\item[(ii)] If the family $\mathcal{G}_{\delta}(\mathbb{R})$ is stable under countable intersections, then every denumerable family of non-empty cuc subsets of $\mathbb{R}$ has a choice function, which implies $\mathbf{CAC}_{\omega}(\mathbb{R})$.
	\end{enumerate}
\end{theorem}
\begin{proof}
	(i) Let us assume that $\mathbf{CAC}(\mathbb{R})$ holds. Let $\tau$ be the natural topology of $\mathbf{X}$. Then $\tau$ is equipotent to a subset $\mathbb{R}$. Let $\{G_n: n\in\omega\}$ be a family of members of $\mathcal{G}_{\delta}(\mathbf{X})$. For every $n\in\omega$, let $\mathcal{V}_n= \{\psi\in \tau^{\omega}: G_n=\bigcap\limits_{k\in\omega}\psi(k)\}$. Since $|\mathbb{R}^\omega|=|\mathbb{R}|$ and $|\tau|\leq |\mathbb{R}|$, it follows that $\tau^{\omega}$ is equipotent to a subset of $\mathbb{R}$. Hence, it follows from $\mathbf{CAC}(\mathbb{R})$ that there exists a family $\{\psi_n: n\in\omega\}$ such that, for every $n\in\omega$, $\psi_n\in\mathcal{V}_n$.  The family $\mathcal{E}=\{\psi_n(k): n,k\in\omega\}$ is countable, $\mathcal{E}\subseteq \tau$ and $G=\bigcap\mathcal{E}$. Hence $\mathcal{G}_{\delta}(\mathbf{X})$ is stable under countable intersections.\medskip
	
	(ii) Let us fix a family $\mathcal{A}=\{A_n: n\in\mathbb{N}\}$ of non-empty cuc subsets of $\mathbb{R}$. Since, for every $n\in\mathbb{N}$, we can fix a homeomorphism of $\mathbb{R}$ onto the interval $(n, n+1)$, without loss of generality, we may assume that, for every $n\in\mathbb{N}$, $A_n\subseteq (n, n+1)$. Let $A=\bigcup\limits_{n\in\mathbb{N}}A_n$. Assume that the family $\mathcal{G}_{\delta}(\mathbb{R})$ is stable under countable intersections. Then the family $\mathcal{F}_{\sigma}(\mathbb{R})$ is stable under countable unions. This implies that, for every $n\in\mathbb{N}$, $A_n\in\mathcal{F}_{\sigma}(\mathbb{R})$ and, in consequence, $A\in\mathcal{F}_{\sigma}(\mathbb{R})$. Hence, there exists a family $\{F_m: m\in\mathbb{N}\}$ of closed sets of $\mathbb{R}$ such that $A=\bigcup\limits_{m\in\mathbb{N}}F_m$. For every pair $n,m\in\mathbb{N}$, we have $F_m\cap [n, n+1]\subseteq A_n$, and the set $F_m\cap [n, n+1]$ is closed in $\mathbb{R}$. For every $n\in\mathbb{N}$, we can define $m(n)=\min\{m\in\mathbb{N}: F_m\cap [n, n+1]\neq\emptyset\}$. Since $\mathbb{R}$ is Loeb, there exists $f\in\prod\limits_{n\in\mathbb{N}}(F_{m(n)}\cap [n, n+1])$. Then $f$ is a choice function of $\mathcal{A}$. 
\end{proof}

It is known that $\mathbf{CAC}_{\omega}(\mathbb{R})$ is false, for instance, in Sageev's Model I (see model $\mathcal{M}6$ in \cite{hr}). Therefore, the following corollary can be deduced from Theorem \ref{s5t1}(ii):

\begin{corollary}
	\label{s5c2} 
	The statement ``$\mathcal{G}_{\delta}(\mathbb{R})$ is stable under countable intersections''  is unprovable in $\mathbf{ZF}$.
\end{corollary}

Of course, if $\mathbf{X}$ is a metrizable space, then the family $\mathcal{Z}(\mathbf{X})$  is stable under countable intersections. 

\begin{theorem}
	\label{s5t3}
	$[\mathbf{ZF}]$
	\begin{enumerate}
		\item[(i)] $\mathbf{CMC}\rightarrow(\mathbf{G}(T_{3\frac{1}{2}},\mathcal{G}_{\delta\delta},\mathcal{G}_{\delta})\wedge\mathbf{Z}(CR, \mathcal{Z}_{\delta},\mathcal{Z})).$
		\item[(ii)] $\mathbf{G}(T_{3\frac{1}{2}},\mathcal{G}_{\delta\delta}, \mathcal{G}_{\delta})\rightarrow (\mathbf{UT}(\aleph_0, cuc, cuc)\wedge\mathbf{UT}(\aleph_0, cuf, cuf))$.
		\item[(iii)] $\mathbf{Z}(CR, \mathcal{Z}_{\delta},\mathcal{Z}))\rightarrow(\mathbf{UT}(\aleph_0, cuc, cuc)\wedge\mathbf{UT}(\aleph_0, cuf, cuf))$.
	\end{enumerate}
\end{theorem}

\begin{proof}
	That (i) holds follows from Theorem \ref{s4t9}(i). 
	
	To prove (ii) and (iii), let us suppose that $\mathbf{G}(T_{3\frac{1}{2}},\mathcal{G}_{\delta\delta}, \mathcal{G}_{\delta})$ (respectively, $\mathbf{Z}(CR, \mathcal{Z}_{\delta},\mathcal{Z})$) is true and fix a family $\mathcal{A}=\{A_n: n\in\omega\}$ of non-empty cuc sets (respectively, cuf sets).  Let $A=\bigcup\mathcal{A}$.  We can fix a sequence $(\infty_n)_{n\in\omega}$ of pairwise distinct elements such that, for every $n\in\omega$, $\infty_n\notin A$. For every $n\in\omega$, we put $X_n=A_n\cup\{\infty_n\}$ and $\tau_n=\mathcal{P}(A_n)\cup\{X_n\setminus C: C\in[A_n]^{\leq\omega}\}$ (respectively, $\tau_n=\mathcal{P}(A_n)\cup\{X_n\setminus C: C\in[A_n]^{<\omega}\}$). One can easily check that, for every $n\in\omega$,  $\tau_n$ is a topology on $X_n$ such that the space $\mathbf{X}_n=\langle X_n,\tau_n\rangle$ is a zero-dimensional $T_1$-space, so $\mathbf{X}_n$ is a $T_{3\frac{1}{2}}$-space. Let $\mathbf{X}=\bigoplus\limits_{n\in\omega}\mathbf{X}_n$. Since every direct sum of $T_{3\frac{1}{2}}$-spaces is a $T_{3\frac{1}{2}}$-space, $\mathbf{X}$ is a $T_{3\frac{1}{2}}$-space. For every $n\in\omega$, the set $A_n$ is a cuc set, so $A_n$ is of type $\mathcal{F}_{\sigma}$ in $\mathbf{X}$. In fact, for every $n\in\omega$, the set $A_n$ is a countable union of clopen sets in $\mathbf{X}$, so $A_n$ is a co-zero set of $\mathbf{X}$.  Since $\mathcal{G}_{\delta}(\mathbf{X})$ (respectively, $\mathcal{Z}(\mathbf{X})$) is stable under countable intersections, the set $A$ is of type $\mathcal{F}_{\sigma}$ (respectively, $A$ is a co-zero set) in $\mathbf{X}$. Hence, we can fix a family $\{H_i: i\in\omega\}$ of closed sets of $\mathbf{X}$ such that $A=\bigcup\limits_{i\in\omega}H_i$. For every pair $i,n$ of members of $\omega$, the set $H_{i,n}=H_i\cap A_n$ is closed in $\mathbf{X}_n$, which implies that $H_{i,n}$ is countable (respectively, finite).  Then $A=\bigcup\limits_{i,n\in\omega}H_{i,n}$ is a cuc set (respectively, cuf set),  Hence $\mathbf{G}(T_{3\frac{1}{2}},\mathcal{G}_{\delta\delta}, \mathcal{G}_{\delta})$ (respectively, $\mathbf{Z}(CR, \mathcal{Z}_{\delta},\mathcal{Z}))$) implies both $\mathbf{UT}(\aleph_0, cuc, cuc)$ and $\mathbf{UT}(\aleph_0, cuf, cuf)$.
\end{proof}

\begin{remark}
	\label{s5r4}
	It is known, for instance, from \cite{chhkr} that it holds in $\mathbf{ZF}$ and in $\mathbf{ZFA}$ that $\mathbf{MC}$ implies $\mathbf{UT}(\aleph_0, cuf, cuf)$. It was shown in \cite[Theorem 3.1]{ht} that, in $\mathbf{ZF}$,  $\mathbf{CMC}$ implies $\mathbf{UT}(\aleph_0, cuf, cuf)$. Our proof of Theorem \ref{s5t3} is an alternative proof that $\mathbf{CMC}$ implies $\mathbf{UT}(\aleph_0, cuf, cuf)$ in $\mathbf{ZF}$.
\end{remark}

\begin{remark}
	\label{s5r5}
	It was noticed in \cite[Theorem 3.3]{chhkr} that $\mathbf{UT}(\aleph_0, cuf, cuf)$ implies $\mathbf{CMC}_{\omega}$. Hence, it follows from Theorem \ref{s5t3} that the following implications are true in $\mathbf{ZF}$:
	\begin{enumerate}
		\item[(i)] $\mathbf{G}(T_{3\frac{1}{2}},\mathcal{G}_{\delta\delta}, \mathcal{G}_{\delta})\rightarrow\mathbf{CMC}_{\omega}$;
		\item[(ii)] $\mathbf{Z}(CR, \mathcal{Z}_{\delta},\mathcal{Z})\rightarrow \mathbf{CMC}_{\omega}$.
	\end{enumerate}
\end{remark}

The following proposition is obvious.

\begin{proposition}
	\label{s5p6}
	$[\mathbf{ZF}]$
	\begin{enumerate}
		\item[(i)] $\mathbf{PPP}$ if and only if every finite product of $P$-spaces is a $P$-space.
		\item[(ii)] $\mathbf{ABP}\rightarrow\mathbf{AB0P}$.
		\item[(iii)] $\mathbf{WBP}\rightarrow\mathbf{WB0P}$.
		\item[(iv)] $(\mathbf{WBP}\rightarrow\mathbf{ABP})\wedge(\mathbf{WB0P}\rightarrow\mathbf{AB0P})$.
		\item[(v)] $\mathbf{SBSqB}\rightarrow\mathbf{SB0SqB}$.
	\end{enumerate}
\end{proposition}

\begin{remark}
	\label{s5r7} 
	In $\mathbf{ZF}$, the Cantor cube $\mathbf{2}^{\omega}$ can serve as an example of a countable product of $P$-spaces which is not a $P$-space. At this moment, we do not know if, in a model of $\mathbf{ZF}$, a finite product of $P$-spaces may fail to be a $P$-space.  
\end{remark}
It was noticed in \cite[proof of Corollary 5.8]{gh} that if $J$ is an infinite set and $\{\mathbf{X}_j: j\in J\}$ is a family of non-empty finite discrete spaces, then the product $\prod\limits_{j\in J}\mathbf{X}_j$ is not a $P$-space in $\mathbf{ZFC}$. That this $\mathbf{ZFC}$-result of \cite{gh} is unprovable in $\mathbf{ZF}$ follows from the following theorem and the fact that there are models of $\mathbf{ZF}$ in which $\mathbf{CAC}_{fin}$ fails.

\begin{theorem}
	\label{s5t8}
	$[\mathbf{ZF}]$
	\begin{enumerate} 
		\item[(i)]$\mathbf{CAC}$ if and only if, for every family $\{\mathbf{X}_n: n\in\omega\}$ of discrete spaces, it holds that if, for every $n\in\omega$, $\mathbf{X}_n$ consists of at least two points, then the product $\prod_{n\in\omega}\mathbf{X}_n$  is not a $P$-space.
		\item[(ii)] $\mathbf{CAC}_{fin}$ if and only if, for every family $\{\mathbf{X}_n: n\in\omega\}$ of finite discrete spaces, it holds that if, for every $n\in\omega$, $\mathbf{X}_n$ consists of at least two points, then the product $\prod_{n\in\omega}\mathbf{X}_n$  is not a $P$-space.
	\end{enumerate}
\end{theorem}
\begin{proof} Since the empty space is a $P$-space, it is obvious that the statement ``Every denumerable product of discrete spaces (respectively, finite discrete spaces), each of which consisting of at least two points, is not a $P$-space'' implies $\mathbf{CAC}$ (respectively, $\mathbf{CAC}_{fin}$).
	
	Let us assume $\mathbf{CAC}$ (respectively, $\mathbf{CAC}_{fin}$) and consider a family $\{ X_n: n\in\omega\}$ such that, for every $n\in\omega$, $X_n$ consists of at least two elements (respectively, is a finite set consisting of at least two elements). For every $n\in\omega$, put  $\mathbf{X}_n=\langle X_n, \mathcal{P}(X_n)\rangle$. Let $X=\prod\limits_{n\in\omega}X_n$ and $\mathbf{X}=\prod\limits_{n\in\omega}\mathbf{X}_n$. The space $\mathbf{X}$ is metrizable. It follows from $\mathbf{CAC}$ that $\mathbf{X}$ is dense-in-itself. If $\mathbf{CAC}_{fin}$ holds and, for every $n\in\omega$, the set $X_n$ is finite, then $\mathbf{X}$ is also dense-in-itself. Hence, under our assumptions, $\mathbf{X}$ is not discrete. Since every metrizable $P$-space is discrete, $\mathbf{X}$ is not a $P$-space.
\end{proof}

\begin{theorem}
	\label{s5t9}
	$[\mathbf{ZF}]$
	\begin{enumerate}
		\item[(i)] $\mathbf{CMC}(\aleph_0,\infty)\rightarrow\mathbf{ABP}\rightarrow\mathbf{PPP}$.
		\item[(ii)] $\mathbf{CACCLO}\leftrightarrow\mathbf{PCACCLO}$.
		\item[(iii)]  $\mathbf{AB0P}\rightarrow\mathbf{CACCLO}\rightarrow\mathbf{IDI}(\mathbb{R})$. 
	\end{enumerate}
\end{theorem}
\begin{proof}
	(i) To show that the first implication of (i) is true, we fix any topological space $\mathbf{X}$ which has an almost stable under countable intersections base $\mathcal{B}$. Assuming $\mathbf{CMC}(\aleph_0, \infty)$, we show that $\mathbf{X}$ is a $P$-space. To this aim, we fix any family $\mathcal{U}=\{U_n: n\in\omega\}$ of open sets of $\mathbf{X}$ with $\bigcap\mathcal{U}\neq\emptyset$. Let $x\in\bigcap\mathcal{U}$. For every $n\in\omega$, let $\mathcal{B}_n=\{B\in\mathcal{B}: x\in B\subseteq U_n\}$. By $\mathbf{CMC}(\aleph_0, \infty)$, there exists a family $\{\mathcal{V}_n: n\in\omega\}$ of non-empty countable sets such that, for every $n\in\omega$, $\mathcal{V}_n\subseteq \mathcal{B}_n$. For every $n\in\omega$, let $V_n=\bigcap\mathcal{V}_n$. Clearly, if $n\in\omega$, then $x\in V_n$ and, since $\mathcal{B}$ is almost stable under countable intersections, $V_n\in\mathcal{B}$ and, in consequence, $\bigcap\limits_{n\in\omega}V_n\in\mathcal{B}$. Since, in addition, $x\in\bigcap\limits_{n\in\omega}V_n\subseteq \bigcap\mathcal{U}$, we obtain that $\bigcap\mathcal{U}$ is open in $\mathbf{X}$. Hence $\mathbf{X}$ is a $P$-space. 
	
	To show that the second implication of (i) is also true, we assume $\mathbf{ABP}$ and fix any pair $\mathbf{X}_1, \mathbf{X}_2$ of non-empty $P$-spaces. For $i\in\{1,2\}$, let $\mathbf{X}_i=\langle X_i, \tau_i\rangle$ and $\mathcal{B}_i=\tau_i\setminus\{\emptyset\}$. Then, for $i\in\{1,2\}$, the family $\mathcal{B}_i$ is an almost stable under countable interesctions base of $\mathbf{X}_i$. Let $\mathcal{B}_0=\{U\times V: U\in\mathcal{B}_1\wedge V\in\mathcal{B}_2\}$. It is easily seen that $\mathcal{B}_0$ is an almost stable under countable intersections base of $\mathbf{X}_1\times\mathbf{X}_2$. By $\mathbf{ABP}$, $\mathbf{X}_1\times\mathbf{X}_2$ is a $P$-space.\medskip
	
	(ii) It suffices to show that $\mathbf{PCACCLO}$ implies $\mathbf{CACCLO}$. To this end, we fix a family $\{\langle A_i, \leq_i\rangle: i\in\omega\}$ such that, for every $i\in\omega$, $\leq_i$ is a strongly complete linear ordering on $A_i$. For every $n\in\mathbb{N}$, let $B_n=\prod_{i\in n}A_i$ and let $\preceq_n$ be the lexicographic ordering on $B_n$ corresponding to $\{\leq_i: i\in n\}$. Then, for every $n\in\mathbb{N}$, $\preceq_n$ is a strongly complete linear ordering on $B_n$. Assuming $\mathbf{PCACCLO}$, we can fix an infinite set $M\subseteq\mathbb{N}$ and a function $f\in\prod\limits_{n\in M}B_n$. As usual, for every $i\in\omega$, let $m(i)=\min\{m\in M: i\in m\}$ and $g(i)=f(m(i))(i)$. Then $g\in\prod_{i\in\omega}A_i$.\medskip
	
	(iii) Let us suppose that $\mathbf{IDI}(\mathbb{R})$ is false. Then there exists a Dedekind-finite subset $D$ of $\mathbb{R}$ such that $D$ is dense in $\mathbb{R}$. For every $n\in\mathbb{N}$, let $D_n=(n, n+1)\cap D$ and let $\leq_n$ be the usual linear ordering of $D_n$ inherited from the standard linear ordering $\leq$ of $\mathbb{R}$. Since, for every $n\in\mathbb{N}$, the set $D_n$ is Dedekind-finite, $\leq_n$ is a strongly complete linear ordering on $D_n$. Since $D$ contains no infinite countable subsets, $\prod\limits_{n\in\mathbb{N}}D_n=\emptyset$. This contradicts $\mathbf{CACCLO}$. Hence $\mathbf{CACCLO}$ implies $\mathbf{IDI}(\mathbb{R})$.
	
	Finally, assuming that $\mathbf{PCACCLO}$ is false, let us deduce that $\mathbf{AB0P}$ is also false. To this aim, suppose that $\mathcal{A}^{\ast}=\{\langle A_i, \leq_i\rangle: i\in\omega\}$ is a family of non-empty strongly complete linearly ordered sets such that the family $\mathcal{A}=\{A_i: i\in\omega\}$ is disjoint but does not have a partial choice function. Let $A=\bigcup\limits_{i\in\omega}A_i$ and let $\infty$ be an element such that $\infty\notin{A}$. We put $X=A\cup\{\infty\}$ and define a linear ordering $\leq$ on $X$ as follows. For every $a\in A$, $a\le\infty$. For every $i\in\omega$ and any pair $a,b$ of elements of $A_i$, $a\leq b$ if and only if $a\leq_i b$. For every pair $i,j$ of distinct elements of $\omega$, any $a\in A_i$ and $b\in A_j$, $a\leq b$ if and only if $i\in j$. Next, we define
	$$\mathcal{B}=\{\{x\}: x\in A\}\cup\{ [x, \infty]: x\in A\}$$
	where $[x, \infty]=\{y\in X: x\leq y\leq\infty\}$. We denote by $\tau$ the topology on $X$ such that $\mathcal{B}$ is a base of the space $\mathbf{X}=\langle X, \tau\rangle$. It is easily seen that $\mathbf{X}$ is a zero-dimensional $T_1$-space. Let us prove that the base $\mathcal{B}$ is almost stable under countable intersections. 
	
	Let $\mathcal{U}$ be a denumerable family of subsets of $\mathcal{B}$ such that the set $G=\bigcap\mathcal{U}$ is non-empty. We are going to show that $G\in\mathcal{B}$. We notice that at most one member of $\mathcal{U}$ can be a singleton. If there exists $U_0\in\mathcal{U}$ such that $U_0$ is a singleton, then $G=U_0\in\mathcal{B}$. Suppose that none of the members of $\mathcal{U}$ are singletons. Then, for every $U\in\mathcal{U}$, there is $y_U\in A$ such that $U=[y_U, \infty]$. The set $Y=\{y_U: U\in\mathcal{U}\}$ is a denumerable subset of $A$. Since $\mathcal{A}$ has no partial choice function, there exists $z\in A$ such that, for every $U\in\mathcal{U}$, $y_U\leq z$. Since, for every $i\in\omega$, $\leq_i$ is a strongly complete linear ordering on $A_i$, there exists $z_0\in A$ such that $z_0=\sup\{y_U: U\in\mathcal{U}\}$ in $\langle X, \leq\rangle$. Then $G=[z_0, \infty]\in\mathcal{B}$. This shows that $\mathcal{B}$ is almost stable under countable intersections. Using similar arguments, one can show that $\mathcal{B}$ is even strongly stable under countable intersections. Now, we notice that, for every $i\in\omega$, the set $V_i=\bigcup\{[x, \infty]: x\in A_i\}$ is open in $\mathbf{X}$; however, the set $\bigcap\limits_{i\in\omega}V_i=\{\infty\}$ is not open in $\mathbf{X}$, so $\mathbf{X}$ is not a $P$-space. This completes the proof of (iii).
\end{proof}

\begin{theorem}
	\label{s5t10}
	$[\mathbf{ZF}]$
	\begin{enumerate}
		\item[(i)] $\mathbf{DC}\rightarrow\mathbf{SBSqB}\rightarrow\mathbf{SB0SqB}\rightarrow\mathbf{CAC}$.
		\item[(ii)] $(\mathbf{ABP}\wedge\mathbf{Z}(CR,\mathcal{Z}_{\delta},\mathcal{Z}))\rightarrow \mathbf{CRZOP}$.
		\item[(iii)] $\mathbf{WB0P}\rightarrow (\mathbf{CAC}(\mathbb{R})\wedge\mathbf{vDCP}(\omega))$.
	\end{enumerate}
\end{theorem}
\begin{proof}
	(i) To show that $\mathbf{DC}$ implies $\mathbf{SBSqB}$, we suppose that $\mathbf{X}$ is a topological space which has a base $\mathcal{B}$ strongly stable under countable intersections, and we prove that $\mathbf{DC}$ implies that $\mathbf{X}$ is even strongly Baire. To this aim, we consider a family $\mathcal{D}=\{D_n: n\in\omega\}$ of dense open sets of $\mathbf{X}$ and a set $W\in\mathcal{B}$. Let $D=\bigcap\limits_{n\in\omega}D_n$ and, for every $n\in\omega$, let $\mathcal{D}_n=\{B\in\mathcal{B}: B\subseteq D_n\cap W\}$. We notice that, for every $n\in\omega$, the set $\mathcal{D}_n$ is dense in the partially ordered set $\langle\mathcal{B}, \subseteq\rangle$. Hence,  by $\mathbf{DC}$ and  \cite[Form \text{[43 C]}]{hr}, there exists a sequence $(B_n)_{n\in\omega}$ of members of $\mathcal{B}$ such that, for every $n\in\omega$, $B_{n+1}\subseteq B_n\in \mathcal{D}_n$. Let $U=\bigcap\limits_{n\in\omega}B_n$. Since $\mathcal{B}$ is strongly stable under countable intersections, we have $U\in\mathcal{B}$. Hence $U\neq\emptyset$. Since $U\subseteq D\cap W$, the set $D$ is dense in $\mathbf{X}$. To see that $D$ is open in $\mathbf{X}$, we fix a point $x\in D$. By $\mathbf{DC}$ and \cite[Form \text{[43 C]}]{hr}, there exists a sequence $(B_{n,x})_{n\in\omega}$ of members of $\mathcal{B}$ such that, for every $n\in\omega$, $x\in \mathcal{B}_{n+1,x}\subseteq B_{n,x}\in\{B\in\mathcal{B}: x\in B\subseteq D_n\}$. Then $x\in\bigcap\limits_{n\in\omega}B_{n,x}\subseteq D$ and $\bigcap\limits_{n\in\omega}B_{n,x}\in\mathcal{B}$. In consequence, $D$ is open in $\mathbf{X}$. Hence the first implication of (i) is true. The second implication of (i) is obvious (see Proposition \ref{s5p6}(v)). 
	
	To show that $\mathbf{SB0SqB}$ implies $\mathbf{CAC}$, suppose that $\mathbf{CAC}$ is false. Then there exists a family $\mathcal{A}=\{A_n: n\in\omega\}$ of infinite sets such that $\mathcal{A}$ has no partial choice function. Let $\infty$ be an element such that $\infty\notin\bigcup\mathcal{A}$. For every $n\in\omega$, let $X_n= A_n\cup\{\infty\}$ and $\mathbf{X}_n=\mathbf{A}_n(\infty)$. For every $n\in\omega$, let $\mathcal{B}_n=\{\{x\}: x\in A_n\}\cup\{ X_n\setminus F: F\in[A_n]^{<\omega}\}$. Let $X=\prod\limits_{n\in\omega}X_n$ and $\mathbf{X}=\prod\limits_{n\in\omega}\mathbf{X}_n$. Then $\mathbf{X}$ is a zero-dimensional $T_1$-space. Let
	$$\mathcal{B}=\{\prod\limits_{n\in\omega}V_n: ((\forall n\in\omega)V_n\in\mathcal{B}_n)\wedge((\exists k\in\omega)(\forall n\in\omega)(k\in n\rightarrow V_n=X_n))\}.$$
	Clearly, $\mathcal{B}$ is a base of $\mathbf{X}$, and $\emptyset\notin\mathcal{B}$. To show that the base $\mathcal{B}$ is strongly stable under countable intersections, we fix a centered family $\mathcal{U}=\{U_i: i\in\omega\}$ of members of $\mathcal{B}$. For every $i\in\omega$, let $U_i=\prod_{n\in\omega}V_{i,n}$ where, for every $n\in\omega$, $V_{i,n}\in\mathcal{B}_n$ and there is $k(i)\in\omega$ such that, for every $n\in\omega$, if $k(i)\in n$, then $V_{i,n}=X_n$. For every $i\in\omega$, let 
	$$E_i=\{n\in\omega: (\exists x\in A_{n}) V_{i,n}=\{x\}\}.$$ 
	Let $E=\bigcup\limits_{i\in\omega}E_i$. Since $\mathcal{U}$ is centered, it follows that, for every $n\in E$, the family $\{V_{i,n}: i\in\omega\}$ is also centered. Hence, for every $n\in E$, there exists a unique $x(n)\in A_n$ such that, for every $i\in\omega$, $V_{i,n}=\{x(n)\}$. Thus, if $E$ is infinite, by assigning to every $n\in E$ the point $x(n)$, we obtain a partial choice function of $\mathcal{A}$. This is impossible. In consequence, the set $E$ is finite. Therefore, if $U=\bigcap\limits_{i\in\omega}U_i$, then $U=\prod\limits_{n\in\omega}W_n$ where, for every $n\in E$, $W_n=\{x(n)\}$ and, for every $n\in\omega\setminus E$, $W_n=X_n$. Hence $U\in\mathcal{B}$. This proves that $\mathcal{B}$ is strongly stable under countable intersections. Of course, for every $n\in\omega$, the set $\pi_n^{-1}[A_n]$ is open and dense in $\mathbf{X}$. If $\mathbf{SB0qB}$ holds, then $\bigcap\limits_{n\in\omega}\pi_n^{-1}[A_n]\neq\emptyset$. However, every element of $\bigcap\limits_{n\in\omega}\pi_n^{-1}[A_n]$ is a choice function of $\mathcal{A}$. This  proves that $\mathbf{SB0qB}$ implies $\mathbf{CAC}$.\medskip
	
	(ii) Assume that $\mathbf{ABP}$ and $\mathbf{Z}(CR,\mathcal{Z}_{\delta},\mathcal{Z}))$ are both true. Let $\mathbf{X}$ be a completely regular space such that every zero-set of $\mathbf{X}$ is open in $\mathbf{X}$. Then the family $\mathcal{Z}(\mathbf{X})$ is a base of $\mathbf{X}$. By $\mathbf{Z}(CR,\mathcal{Z}_{\delta},\mathcal{Z}))$, the base $\mathcal{Z}(\mathbf{X})$ is stable under countable intersections. This, together with $\mathbf{ABP}$, implies that $\mathbf{X}$ is a $P$-space.\medskip
	
	(iii) Let us show that $\mathbf{WB0P}$ implies $\mathbf{CAC}(\mathbb{R})$. To this aim, we apply \cite[Theorem 3.14]{kw0} asserting that $\mathbf{CAC}(\mathbb{R})$ is equivalent to the statement: ``Every denumerable family of dense subsets of $\mathbb{R}$ has a partial choice function''. Hence, assuming that $\mathbf{CAC}(\mathbb{R})$ is false, we can fix a family $\mathcal{A}=\{A_n: n\in\omega\}$ of subsets of $\mathbb{R}$ such that $\mathcal{A}$ has no partial choice function and, for every $n\in\omega$, $A_n$ is a dense subspace of the interval $(n, n+1)$ equipped with the natural topology inherited from $\mathbb{R}$. We choose an element $\infty\notin\bigcup\mathcal{A}$. We put $X=\bigcup\mathcal{A}\cup\{\infty\}$. We extend the standard linear order $\leq$ on $\bigcup\mathcal{A}$ inherited from $\langle\mathbb{R}, \leq\rangle$ to a linear order $\leq$ on $X$ by requiring that, for every $x\in\bigcup\mathcal{A}$, $x<\infty$. Now, we mimic and modify the proof of the first implication of Theorem \ref{s5t9}(iii). Namely, we put
	$$\mathcal{B}=\{\{x\}: x\in\bigcup\mathcal{A}\}\cup\{[x, \infty]: x\in\bigcup\mathcal{A}\}$$
	where, for every $x\in\bigcup\mathcal{A}$,  $[x,\infty]=\{ y\in\bigcup\mathcal{A}: x\leq y\leq\infty\}$.  We consider the topology $\tau$ on $X$ such that $\mathcal{B}$ is a base of the topological space $\mathbf{X}=\langle X, \tau\rangle$. Then $\mathbf{X}$ is a zero-dimensional $T_1$-space. To prove that the base $\mathcal{B}$ of $\mathbf{X}$ is weakly stable under countable intersections, we fix a denumerable family $\mathcal{U}$ of $\mathcal{B}$ such that $\bigcap\mathcal{U}\neq\emptyset$. Clearly, if there exist $U_0\in\mathcal{U}$ and $x_0\in\bigcup\mathcal{A}$ such that $U_0=\{x_0\}$, then $\bigcap\mathcal{U}=U_0\in\mathcal{B}$. Therefore, in much the same way, as in the proof of Theorem \ref{s5t9}(iii), let us suppose that none of the members of $\mathcal{U}$ is a singleton. For every $U\in\mathcal{U}$, let $y_U\in\mathcal{A}$ be such that $U=[y_U, \infty]$. Since $\mathcal{A}$ does not have a partial choice function and $\mathcal{U}$ is denumerable, the set $\{n\in\omega: (\exists U\in\mathcal{U}) y_U\in (n, n+1)\}$ is finite. Hence, there exists $z\in\bigcup\mathcal{A}$ such that, for every $U\in\mathcal{U}$, $y_U\leq z$. Let $z_0=\sup\{y_U: U\in\mathcal{U}\}$ where the supremum is taken in $\mathbb{R}$. We notice that if $z_0\in\bigcup\mathcal{A}$, then $\bigcap\mathcal{U}=[z_0,\infty]\in\mathcal{B}$. Suppose that $z_0\in\mathbb{R}\setminus\bigcup\mathcal{A}$. Let $n_0\in\omega$ be such that $z_0\in (n_0, n_0+1]$. It follows from the density of $A_{n_0}$ in $(n_0, n_0+1)$ that $\bigcap\mathcal{U}=\bigcup\{ [t, \infty]: t\in A_{n_0}\cup A_{n_0+1}\wedge z_0<t\}$. In this case, $\bigcap\mathcal{U}\in\tau$. In consequence, the base $\mathcal{B}$ is weakly stable under countable intersections. Now, we observe that, for every $n\in\omega$, the set $V_n=\bigcup\{[x,\infty]: x\in A_n\}$ is open in $\mathbf{X}$; however, $\bigcap\limits_{n\in\omega}V_n=\{\infty\}\notin\tau$. This shows that $\mathbf{X}$ is not a $P$-space. Hence $\mathbf{WB0P}$ implies $\mathbf{CAC}(\mathbb{R})$. By a slight modification of these arguments and the proof of Theorem \ref{s5t9}(iii), one can show that $\mathbf{WB0P}$ implies $\mathbf{vDCP}(\omega)$. We shall leave the details as an exercise for the interested readers.
\end{proof}

\begin{corollary}
\label{s5c10}
$[\mathbf{ZF}]$  $$(\mathbf{SB0SqB}\rightarrow\mathbf{ABP})\wedge(\mathbf{ABP}\nrightarrow \mathbf{SB0SqB}).$$
\end{corollary}
\begin{proof}
That $\mathbf{SB0SqB}$ implies $\mathbf{ABP}$ follows directly from Theorems \ref{s5t9}(i) and \ref{s5t10}(i). In Pincus' Model II (that is, model $\mathcal{M}29$ in \cite{hr}), $\mathbf{CAC}$ is false and $\mathbf{CMC}(\aleph_0, \infty)$ is true. Therefore, by Theorems \ref{s5t9}(i) and \ref{s5t10}(i), the implication $\mathbf{ABP}\rightarrow \mathbf{SB0SqB}$ is false in $\mathcal{M}29$.
\end{proof}

\begin{remark}
	\label{s5r10}
In view of Theorem \ref{s5t10}(i) and Corollary \ref{s5c10}, $\mathbf{DC}$ implies $\mathbf{ABP}$ in $\mathbf{ZF}$. This can be also shown by arguing similarly to the second part of the proof of Theorem \ref{s5t10}(i). That $\mathbf{ABP}$ does not imply $\mathbf{DC}$ in $\mathbf{ZF}$ follows from Corollary \ref{s5c10}.
\end{remark}

\begin{definition}
	\label{s5d11}
	We say that a topological space $\mathbf{X}$ is \emph{base-hereditarily Lindel\"of} if it has a base $\mathcal{B}$ such that, for every $B\in\mathcal{B}$ and every family $\mathcal{G}\subseteq\mathcal{B}$ such that $B\subseteq\bigcup\mathcal{G}$, there exists a countable $\mathcal{H}\subseteq\mathcal{G}$ such that $B\subseteq\mathcal{H}$.
\end{definition}

The following theorem sheds a little more light on the difficulties with establishing the status of $\mathbf{PPP}$ in $\mathbf{ZF}$.  

\begin{theorem}
	\label{s5t12}
	$[\mathbf{ZF}]$ Let  $\mathbf{Y}=\langle Y, \tau_Y\rangle$ be a $P$-space. 
	\begin{enumerate}
		\item[(i)] If $\mathbf{X}$ is a $P$-space  whose every point has a well-orderable base of neighborhoods, then $\mathbf{X}\times\mathbf{Y}$ is a $P$-space. In particular, if $\mathbf{X}$ is a first-countable $P$-space, then $\mathbf{X}\times\mathbf{Y}$ is a $P$-space.
		
		\item[(ii)] If $\mathbf{X}$ is a base-hereditarily Lindel\"of $P$-space, then $\mathbf{X}\times\mathbf{Y}$ is a $P$-space.
		
		\item[(iii)]  If $X$ is an infinite quasi Dedekind-finite set, and $\infty$ is an element with $\infty\notin X$, then both $\mathbf{X}_{cof}\times\mathbf{Y}$ and $\mathbf{X}(\infty)\times\mathbf{Y}$ are $P$-spaces.
	\end{enumerate}
\end{theorem}
\begin{proof}
	Let $\mathbf{X}=\langle X, \tau_X\rangle$.
	
	(i) Suppose that $\mathbf{X}$ is a $P$-space whose every point has a well-orderable base of neighborhoods. Let $\{U_n: n\in\omega\}$ be a family of open sets of $\mathbf{X}\times\mathbf{Y}$, and suppose that $\langle x,y\rangle\in \bigcap\limits_{n\in\omega}U_n$. For a non-zero ordinal $\alpha$, let $\{W_{\gamma}: \gamma\in\alpha\}$ be a well-ordered base of neighborhoods of $x$ in $\mathbf{X}$. Let $\mathcal{V}(y)=\{V\in\tau_Y: y\in V\}$. For every $n\in\omega$, let $\Gamma_n=\{\gamma\in\alpha: (\exists V\in\mathcal{V}(y)) W_{\gamma}\times V\subseteq U_n\}$. Clearly, for every $n\in\omega$, $\Gamma_n\neq\emptyset$, so we can define $\gamma_{n}=\min\Gamma_n$, and $V_{n}=\bigcup\{ V\in\mathcal{V}(y): W_{\gamma_n}\times V\subseteq U_n\}$. Let $W=\bigcap\limits_{n\in\omega}W_{\gamma_n}$ and $V=\bigcap\limits_{n\in\omega}V_n$. Clearly $x\in W$, $y\in V$ and $W\times V\subseteq\bigcap\limits_{n\in\omega} U_n$.  Since $\mathbf{X}$ and $\mathbf{Y}$ are $P$-spaces, we have $W\in\tau_X$ and $V\in \tau_Y$, hence  $\langle x, y\rangle\in\inter_{\mathbf{X}\times\mathbf{Y}}(\bigcap\limits_{n\in\omega}(U_n))$. This shows that $\mathbf{X}\times\mathbf{Y}$ is a $P$-space.
	
	(ii) Suppose that $\mathbf{X}$ is a base-hereditarily Lindel\"of $P$-space.  We fix a base $\mathcal{B}$ of $\mathbf{X}$ such that, for every $B\in\mathcal{B}$ and every $\mathcal{G}\subseteq\mathcal{B}$ such that $B\subseteq\bigcup\mathcal{G}$, there exists a countable $\mathcal{H}\subseteq\mathcal{G}$ such that $B\subseteq\bigcup\mathcal{G}$. For $x\in X$ and $y\in Y$, let $\mathcal{W}(x)=\{W\in\mathcal{B}: x\in W\}$ and $\mathcal{V}(y)=\{V\in\tau_Y: y\in V\}$. We fix a point $\langle x_0, y_0\rangle\in X\times Y$ and prove that $\langle x_0, y_0\rangle$ is a $P$-point of $\mathbf{X}\times\mathbf{Y}$.
	
	Let $\mathcal{U}=\{U_n: n\in\omega\}$ be a family of open subsets of $\mathbf{X}\times\mathbf{Y}$  such that $\langle x_0, y_0\rangle\in U=\bigcap\limits_{n\in\omega}U_n$. For every $n\in\omega$, we define $A_n=(X\times\{y_0\})\cap U_n$ and $O_n=\pi_X[A_n]$. Then, for every $n\in\omega$, $x_0\in O_n\in\tau_X$. Let $O=\bigcap\limits_{n\in\omega}O_n$.  Since $\mathbf{X}$ is a $P$-space, we have $O\in\tau_X$, so we can fix $B\in\mathcal{B}$ such that $x_0\in B\subseteq O$.  Let us consider an arbitrary $n\in\omega$ and the family $\mathcal{V}_n=\{V\in\mathcal{V}(y_0): B\times V\subseteq U_n\}$. To prove that $\mathcal{V}_n\neq\emptyset$, we define the family $\mathcal{G}_n=\{ G\in\mathcal{B}: (\exists V\in\mathcal{V}(y_0)) G\times V\subseteq U_n\}$. We notice that $B\subseteq\bigcup\mathcal{G}_n$. Hence, by our assumption about $\mathcal{B}$, there exists a countable subfamily $\mathcal{H}\subseteq\mathcal{G}_n$ such that $B\subseteq\bigcup\mathcal{H}$. For every $G\in\mathcal{H}$, let $V(G)=\bigcup\{ V\in\mathcal{V}(y_0): G\times V\subseteq U_n\}$. Then, for every $G\in\mathcal{H}$,  $V(G)\in\mathcal{V}(y_0)$ and $G\times V(G)\subseteq U_n$. Therefore, since $\mathbf{Y}$ is a $P$-space and $\mathcal{H}$ is countable, we have $\bigcap\{V(G): G\in\mathcal{H}\}\in\mathcal{V}(y_0)$. Clearly, $B\times (\bigcap\{V(G): G\in\mathcal{H}\})\subseteq U_n$, so $\bigcap\{V(G): G\in\mathcal{H}\}\in\mathcal{V}_n$. We define $V_n=\bigcup\mathcal{V}_n$. Since $\mathcal{V}_n\neq\emptyset$, we have $V_n\in\mathcal{V}(y_0)$ and $B\times V_n\subseteq U_n$. Let $V(y_0)=\bigcap\limits_{n\in\omega}V_n$. Since $\mathbf{Y}$ is a $P$-space, we have $V(y_0)\in\mathcal{V}(y_0)$. Clearly, $\langle x_0\ y_0\rangle\in B\times V(y_0)\subseteq U$, which shows that $\langle x_0, y_0\rangle\in\inter_{\mathbf{X}\times\mathbf{Y}}(U)$. Thus, $\langle x_0, y_0\rangle$ is a $P$-point of $\mathbf{X}\times\mathbf{Y}$.
	
	(iii) Finally, suppose that $X$ is an infinite quasi Dedekind-finite set. By Theorem \ref{s4t7}, the spaces $\mathbf{X}_{cof}$ and $\mathbf{X}(\infty)$ are both $P$-spaces. It is obvious that the spaces $\mathbf{X}_{cof}$ and $\mathbf{X}(\infty)$ are also base-hereditarily Lindel\"of. Hence (iii) follows from (ii).
\end{proof}

\begin{remark}
	\label{s5r13}
	In much the same way, as in the proof of Theorem \ref{s5t12}(ii), one can show that if $\mathbf{X}$ is a locally compact regular space, then, for every $P$-space $\mathbf{Y}$, the space $\mathbf{X}\times\mathbf{Y}$ is a $P$-space.
\end{remark}

\section{A shortlist of open problems}
\label{s6}
So many open problems on $P$-spaces or (strongly) zero-dimensional spaces in $\mathbf{ZF}$ can be posed that it seems impossible to make a complete list of them. Therefore, let us pay attention only to the following open problems, strictly related to some of the results included in this article:

\begin{enumerate}
	\item Is there a model of $\mathbf{ZF}$ or a permutation model in which $\mathbf{PPP}$ fails? (See Question \ref{s1q21}.)
	\item Is it provable in $\mathbf{ZF}$ that, for every uncountable set $J$, the Cantor cube $\mathbf{2}^J$ is strongly zero-dimensional?
	\item Are the implications of Theorem \ref{s5t3} reversible in $\mathbf{ZF}$?
	\item Are the implications of items (i) and (iii) of Theorem \ref{s5t9} reversible in $\mathbf{ZF}$?
	\item Which of the implications of Theorem \ref{s5t10}(i)--(iii) is not reversible in $\mathbf{ZF}$?
	\item Are the implications of Proposition \ref{s5p6} (ii)--(v) reversible in $\mathbf{ZF}$?
	\item Does the statement ``For every strongly zero-dimensional space $\mathbf{X}$, $A(\mathbf{X})=C(\mathbf{X})$'' imply $\mathbf{CMC}$? (See Proposition \ref{s3p2}(ii).)
	\item Does the statement ``For every completely regular space $\mathbf{X}$, it holds that $\mathbf{X}$ is strongly zero-dimensional if and only if $U_{\aleph_0}(\mathbf{X})=C(\mathbf{X})$'' imply $\mathbf{CMC}$? (See Proposition \ref{s3p2}(iii).)
	\item Does the statement ``For every topological space $\mathbf{X}$, at least one of conditions (i)--(iv) of Theorem \ref{s4t12} holds'' imply $\mathbf{CMC}$?
	\item Are the implications of Theorem \ref{s4t14} reversible in $\mathbf{ZF}$?
\end{enumerate}


\end{document}